\documentclass[reqno, 12pt]{amsart}
\usepackage{mathrsfs}
\usepackage{mathrsfs}
\usepackage{amsfonts}
\usepackage{bbm}
\usepackage{bbm}
\usepackage{mathrsfs}
\usepackage{amssymb,latexsym}
\newtheorem{theorem}{Theorem}

\newtheorem{remark}[theorem]{Remark}

\newtheorem{lemma}[theorem]{Lemma}
\newtheorem{definition}[theorem]{Definition}

\begin{document}
\setlength{\baselineskip}{1.2\baselineskip}

\title [Dynamical behavior of a system modeling wave bifurcations]
{Dynamical behavior of a system modeling wave bifurcations with higher order viscosity}

\author{Tong Li$^\star$}
\address{$^\star$Department of Mathematics\\
        The University of Iowa, Iowa City\\
         IA, 52242, USA}
\email{tong-li@uiowa.edu}

\author{Xiaoyan Wang$^\dag$}
\address{$^\dag$Department of Mathematics\\
        Indiana University, Bloomington\\
         IN, 47405, USA}
\email{wang264@indiana.edu}

\author{Jinghua Yao$^\ddag$}
\address{$^\ddag$ Corresponding author, Department of Mathematics\\
        The University of Iowa, Iowa City\\
         IA, 52242, USA}
\email{jinghua-yao@uiowa.edu}

\maketitle

\begin{abstract}

We rigorously show that a class of systems of partial differential equations modeling wave bifurcations supports stationary equivariant bifurcation dynamics through deriving its full dynamics on the center manifold(s). A direct consequence of our result is that the oscillations of the dynamics are \textit{not} due to rotation waves though the system exhibits Euclidean symmetries.  The main difficulties of carrying out the program are: 1) the system under study contains multi bifurcation parameters and we do not know \textit{a priori} how they come into play in the bifurcation dynamics. 2) the representation of the linear operator on the center space is a $2\times 2$ zero matrix, which makes the characteristic condition in the well-known normal form theorem trivial. We overcome the first difficulty by using projection method.
We managed to overcome the second subtle difficulty by using a conjugate pair coordinate for the center space and applying duality and projection arguments. Due to the specific complex pair parametrization, we could naturally get a form of the center manifold reduction function, which makes the study of the current dynamics on the center manifold possible. The symmetry of the system plays an essential role in excluding the possibility of bifurcating rotation waves. 

{\bf Keywords:} Spectrum, resolvent, equivariant bifurcation, center manifold, symmetry; Implicit function theorem.\\

{\bf Mathematics Subject Classification(2010):} 37G; 35P; 34G; 34B

\end{abstract}

\maketitle

\section{Introduction}
\setcounter{equation}{0}
\setcounter{theorem}{0}

The current work is a continuation or more precisely a complimentary study of our former works \cite{Yao, LY} on dynamical behaviors of solutions to partial differential equations with symmetries. Here our primary goals include the study of a new bifurcation mechanism based on a different spectral scenario, the demonstration of a second viewpoint of computing the dynamics on center manifold, the illustration of the use of symmetry, and a comparison between the current work and our former works on equivariant Hopf bifurcations. More importantly, we assert that though there are still oscillations in our current dynamics and though the system exhibits Euclidean symmetry, more precisely, the system is $O(2)$-equivariant, the oscillations are not due to bifurcated rotation waves. We actually rigorously show that the oscillations are associated with a stationary equivariant bifurcation and the resulted $S^1$-family of waves can be associated with arbitrary nonzero wave numbers.  The current study and our former studies in \cite{Yao, LY} cover the generic bifurcation dynamics for the system \eqref{system} below, demonstrate that this model for bifurcating waves has abundant interesting dynamics and also shed light on our further investigation of viscous shock waves bifurcations (for this, see the discussion below on future work). One common feature in our work here and \cite{Yao, LY} is that we use Fourier analysis to decompose the problem to get admissible critical configuration points for our bifurcation analysis. This feature has similarity with E. Hopf's study \cite{Ho} on fluid turbulence and with L. Rayleigh's study \cite{Ra} on fluid stability. Meanwhile, we hope that our work could also shed some light on the vanishing viscosity method for which $a, \varepsilon\rightarrow 0^+$. With these goals in mind, we will adopt a similar organizational structure as in our former works for the ease of comparisons. 

\subsection{The system} Let us first recall the system we are going to study and some physical background (see also \cite{Yao, LY}). In this work, we continue to study the following class of systems in continuum mechanics
\begin{equation}\label{system}
\begin{cases}
\partial_t \tau-\partial_x u=-a\partial^4_x\tau,\\
\partial_t u-\partial_x\sigma(\tau)=-\delta\partial^2_x
u-\varepsilon\partial^4_x u
\end{cases}
\end{equation}
on the spatial periodic domain $\mathbb T^1=\mathbb{R}^1/[-M, M]$ where $M$ is any positive constant.
In system (\ref{system}), $\tau=\tau(x,t)$ and $u=u(x,t)$ are the scalar unknown functions. The scalar function $\sigma(\tau)$ is usually called
flux function in mathematics and pressure law in physics. The three parameters $a, \delta,\varepsilon$ are called diffusion coefficients. We emphasize that the domain is $\mathbb{R}^1/[-M, M]$, which means that we have periodic boundary conditions.

Without loss of generality, we can always consider the case $M=\pi$ and $\varepsilon=1$ or 
else we can do the  following two successive scalings and renamings given by
$$t\mapsto \bar{t}=\frac{\pi}{M}t,\,x\mapsto \bar{x}=\frac{\pi}{M}x,\, a\mapsto\bar{a}= \frac{\pi^3}{M^3}a,\, \delta\mapsto\bar{\delta}=\frac{\pi}{M}\delta,\, \varepsilon\mapsto\bar{\varepsilon}= \frac{\pi^3}{M^3}\varepsilon.$$
and 
 $$t\mapsto\tilde t=\varepsilon t, x\mapsto\tilde x=x, u\mapsto \tilde u=u, \tau\mapsto  \tilde \tau=\varepsilon \tau,$$ 
 $$a\mapsto\tilde a= \varepsilon^{-1}a, \sigma(\tau)\mapsto\tilde \sigma(\tilde\tau)=
 \varepsilon^{-1}\sigma(\varepsilon^{-1}\tilde\tau)=\varepsilon^{-1}\sigma(\tau), 
 \delta\mapsto\tilde\delta=\varepsilon^{-1}\delta.$$ 
  
 Meanwhile, we observe that any constant state $(\tau_0, u_0)$ is a solution to system (\ref{system}).
We consider the constant state $(0, 0)$ without loss of generality. This is obvious by first choosing 
$(\tau_0, u_0)$ in the physical range and then using the invariance of system (\ref{system}) in the 
translation group actions $u\mapsto u+h$ for any $h\in \mathbb{R}^1$ and redefining 
$\sigma(\tau)$ by $\sigma(\tau_0+\tau)$. The above constant states are usually referred to as uniform states or
homogeneous states in physics literature.

Systems of form \eqref{system} are generic in classical continuum mechanics (\cite{Ba, Da, EW} 
and the references therein) as they describe 
Newton's Second Law of motion, in gas dynamics, for example the $p$-system (see in particular Chapter 2 of 
Dafermos \cite{Da} and Nishida \cite{Ni}). When $M=\infty$, systems of form \eqref{system} are
also connected with the Kuramoto-Sivashinsky and related systems when we seek traveling waves 
solutions (see \cite{EW, Go, KS, KT, Si, SM}).

If we regard the term $\begin{pmatrix}-a\partial_x^4 \tau\\ -\varepsilon\partial_x^4 u\end{pmatrix}$ as vanishing viscosity term in the system
\eqref{system} and study the process $a\rightarrow 0+,\varepsilon\rightarrow 0+$, this is the vanishing viscosity method in partial differential equations for the well-known $p$-system in gas dynamics.
Of course, this requires $a>0$ and $\varepsilon>0$. However, in our work here, we are not restricted to the case with $a>0$ and $\varepsilon>0$ which gives dissipation. We will see later
that as long as $(a_c, \delta_c)$ is an admissible critical configuration point, we can get bifurcation dynamics under necessary conditions to be specified later. If there are admissible critical configuration point $(a_c,\delta_c)$ such that $a_c>0$ after scalings and renaming, we could scale back and analyze the vanishing viscosity limit process.  Our bifurcation results in particular indicate that in such cases we should avoid the parameter ranges which lead to bifurcations during the limiting process when we use the vanishing viscosity method.

\subsection{Main Results}

In order to state our main results clearly, we first introduce the admissible critical configuration set $\mathcal A(k_0)$ for an arbitrary nonzero integer $k_0$. Actually, we will see that $\mathcal A(k_0)=\mathcal A(-k_0)$. Hence we may regard $k_0$ as a positive integer in the remaining of the paper.

\begin{definition}\label{def}
For a given positive integer $k_0$, we say that the pair $(a_c,\delta_c)\in \mathbbm R^2$ is an admissible critical configuration point if the following three conditions are fulfilled:

(a) $ a_c k_0^4(k_0^2-\delta_c )+\sigma'(0)=0$.

(b) for any nonzero integer $k$ such that $ |k|\neq |k_0|$ and nonzero real number $\omega$,
$$[a_c k^4(k^4-\delta_c k^2)+\sigma'(0)-\omega^2]+i[(a_c+1)k^4-\delta_c k^2]\neq 0.$$

(c) $(a_c+1)k_0^2-\delta_c\neq 0$.


\noindent Consequently, $\mathcal A(k_0)$ is the collection of all admissible critical configuration points $(a_c,\delta_c)$ associate with $k_0$.

\end{definition}

The conditions $(a)-(c)$ characterize a spectral scenario of the linear operator $\mathcal L$ (see Section 3) and exactly makes all the denominators appearing later non-vanishing. Now we are in a position to state our main results.

\begin{theorem}\label{thm}
Let $(a_c,\delta_c)$ be any admissible critical configuration point, $\mu=(\delta_c-k_0^2)(a-a_c)+a_c(\delta-\delta_c)$. System \eqref{system} or equivalently $\eqref{perturb}$ admits a center manifold reduction and supports a stationary equivariant $O(2)$-bifurcation from any uniform solution (take $(\tau, u)=(0,0)$ without loss of generality) near $\mu=0$ in the Hilbert space $Y$ consisting of functions in $L^2_{per}(-\pi, \pi)$ with
zero mean over one period when the parameters $a, \delta$ varies around $(a_c, \delta_c)$. The full dynamics on the center manifold(s) has the following form
\begin{equation}\label{rd}
\frac{dA}{dt}=\frac{k_0^4\mu A}{(a_c+1)k_0^2-\delta_c} +\frac{1}{(a_c+1)k_0^2-\delta_c}\Big(\frac{\sigma''(0)^2}{6a_c k_0^4(21k_0^2-5\delta_c)}-\frac{\sigma'''(0)}{2}\Big)
 A^2A^*+O(|A|^5),
\end{equation}
where $A=A(t)\in\mathbb{C}^1$ and the complex conjugate pair $(A(t), A(t)^*)$ gives the coordinate of the center space and represents the dynamics for $(\tau, u)$ in the center manifold and the $O(|A|^5)$ term is given by $\sum_{j=2}^{m}\mathbbm c_j |A|^{2j+1}A+O(|A|^{2m+3})$ with all the $\mathbbm c_j$ real if $\sigma(\cdot)\in C^{2m+1}$ around $0$.
\end{theorem}

In view of the coefficients $\frac{k_0^4}{(a_c+1)k_0^2-\delta_c}$ and $\frac{1}{(a_c+1)k_0^2-\delta_c}\Big(\frac{\sigma''(0)^2}{6a_c k_0^4(21k_0^2-5\delta_c)}-\frac{\sigma'''(0)}{2}\Big)$ of $\mu A$ and $A^2A^*$ term above
and that the product of them is proportional to $\frac{\sigma''(0)^2}{6a_c k_0^4(21k_0^2-5\delta_c)}-\frac{\sigma'''(0)}{2}$ with positive proportional constant, we have the following bifurcation diagram for the above bifurcation dynamics:

\begin{theorem}\label{1.2}
Let $(a_c,\delta_c)$ be any admissible critical configuration point and $\mu$ be as above and small.
Parameterizing the solution by $\mu$, the system \eqref{system} undergoes a supercritical (resp. subcritical) stationary equivariant $O(2)$-bifurcation around $\mu=0$ when $\frac{1}{(a_c+1)k_0^2-\delta_c}\Big(\frac{\sigma''(0)^2}{6a_c k_0^4(21k_0^2-5\delta_c)}-\frac{\sigma'''(0)}{2}\Big)<0$ (resp., $\frac{1}{(a_c+1)k_0^2-\delta_c}\Big(\frac{\sigma''(0)^2}{6a_c k_0^4(21k_0^2-5\delta_c)}-\frac{\sigma'''(0)}{2}\Big)>0$), more precisely, the following properties hold in a neighborhood of $0$ of $\mathbbm R^1$ for sufficiently small $\mu$ when $(a, \delta)$ varies around $(a_c, \delta_c)$:

(i) If $\frac{\sigma''(0)^2}{6a_c k_0^4(21k_0^2-5\delta_c)}-\frac{\sigma'''(0)}{2}<0$ (resp., $\frac{\sigma''(0)^2}{6a_c k_0^4(21k_0^2-5\delta_c)}-\frac{\sigma'''(0)}{2}>0$), \eqref{system} has precisely one trivial equilibrium $U=0$ for $\mu<0$ (resp., $\mu>0$). This equilibrium is stable when $\frac{1}{(a_c+1)k_0^2-\delta_c}\Big(\frac{\sigma''(0)^2}{6a_c k_0^4(21k_0^2-5\delta_c)}-\frac{\sigma'''(0)}{2}\Big)<0$ and unstable when $\frac{1}{(a_c+1)k_0^2-\delta_c}\Big(\frac{\sigma''(0)^2}{6a_c k_0^4(21k_0^2-5\delta_c)}-\frac{\sigma'''(0)}{2}\Big)>0$,

(ii) If $\frac{\sigma''(0)^2}{6a_c k_0^4(21k_0^2-5\delta_c)}-\frac{\sigma'''(0)}{2}<0$ (resp., $\frac{\sigma''(0)^2}{6a_c k_0^4(21k_0^2-5\delta_c)}-\frac{\sigma'''(0)}{2}>0$), \eqref{system} possesses, for $\mu>0$ (resp., $\mu<0$), the trivial solution $U=0$ and a family of nontrivial equilibrium $U=U^{\theta}_{\epsilon}$  parametrized by $\epsilon=\mu^{1/2}$ and phase constant $\theta\in \mathbb R^1/2\pi\mathbb Z$ and the dependence of $U$ on $\epsilon$ is smooth in a neighborhood of $(0,0)$, and $U_{\epsilon}=O(\epsilon)$, i.e., of magnitude $O(|\mu|^{1/2})$. The bifurcated family of equilibria is stable if $\frac{1}{(a_c+1)k_0^2-\delta_c}\Big(\frac{\sigma''(0)^2}{6a_c k_0^4(21k_0^2-5\delta_c)}-\frac{\sigma'''(0)}{2}\Big)<0$ and unstable if $\frac{1}{(a_c+1)k_0^2-\delta_c}\Big(\frac{\sigma''(0)^2}{6a_c k_0^4(21k_0^2-5\delta_c)}-\frac{\sigma'''(0)}{2}\Big)>0$.
\end{theorem}

The quotient of the coefficient of $\mu A$ term and that of $A^2A^*$ term is given by
$k^4_0 \Big(\frac{\sigma''(0)^2}{6a_c k_0^4(21k_0^2-5\delta_c)}-\frac{\sigma'''(0)}{2}\Big)^{-1}$.
Therefore, each bifurcated equilibrium $U^{\theta}_{\epsilon}$ above corresponds to an bifurcated oscillation wave of \eqref{system} given as
\begin{equation}
k_0^2\epsilon\sqrt{\Big(\frac{\sigma''(0)^2}{6a_c k_0^4(21k_0^2-5\delta_c)}-\frac{\sigma'''(0)}{2}\Big)^{-1}}(e^{i\theta}\xi+e^{-i\theta}\xi^*)+O(|\epsilon|^{3}),
\end{equation}
or more precisely,
\begin{equation}
k_0^2\epsilon\sqrt{\Big(\frac{\sigma''(0)^2}{6a_c k_0^4(21k_0^2-5\delta_c)}-\frac{\sigma'''(0)}{2}\Big)^{-1}}\Big(e^{i(\theta+k_0)x}\begin{pmatrix}1\\-ia_c k_0^3\end{pmatrix}+e^{-i(\theta+k_0)x}\begin{pmatrix}1\\ia_c k_0^3\end{pmatrix}\Big)+O(|\epsilon|^{3}).
\end{equation}
The family of bifurcated oscillation waves is homeomorphic to $S^1$ due to the range of the phase constant $\theta$ and forms a closed orbit around the trivial equilibrium. In other words, there are a family of bifurcated limit cycles, denoted by $C_{a,\delta}$, with amplitude
$O(|\mu|^{1/2})=O(|(\delta_c-k_0^2)(a-a_c)+a_c(\delta-\delta_c)|^{1/2})$ in the bifurcation range for \eqref{system} when the parameters $a$ and $\delta$ varies around any admissible critical configuration point. Further, these bifurcated limit cycles form a bell-shape invariant limit set $I_{a,\delta}$ for \eqref{system} when $\mu$ varies. The stability (or instability) and absorbing (or repelling) properties of the limit cycles and the formed invariant limit set are completely determined by $a_c, \delta_c, \sigma$ and characterized by the sign of $\frac{1}{(a_c+1)k_0^2-\delta_c}\Big(\frac{\sigma''(0)^2}{6a_c k_0^4(21k_0^2-5\delta_c)}-\frac{\sigma'''(0)}{2}\Big)$ through Theorem \ref{1.2}-(ii).

We refer the readers to Section \ref{analysis} and Section \ref{last} for more precise results, explanations and related issues. Here we do not seek to list all the results and subtle points in this introductory section as some of them should be demonstrated in the process of computations. However, we shall remark that the oscillations in the current dynamics driven by the system of partial differential equation \eqref{system} is not due to the bifurcation of rotation waves though the system exhibits Euclidean symmetry-the $O(2)$ symmetry. Another remark is that the wave number $k_0$ enters the dynamics of \eqref{system} on the center manifold(s) through $k_0^2$. This is natural due to symmetry of the system and in particular implies that the dynamics of the system \eqref{system} are exactly same for $\pm k_0$. Furthermore, as the sign of $k_0^2$ is always the same for any nonzero wave number $k_0$, we could conclude by Theorem \ref{1.2} that the bifurcation dynamics of \eqref{system}
under the spectral scenario in the current paper are topologically equivalent in the sense of continuous dynamical systems for any nonzero wave number $k_0$. 

In our current study, we get a family of closed orbits or limit cycles $C_{a,\delta}$ formed by stationary oscillation waves with wave number based on any nonzero integer $k_0$. This mechanism of getting oscillation solutions in our current work is of course different from the well-known Hopf bifurcation mechanism as the spectrum here passes the imaginary axis of the complex plane through the real axis. It is also well-known that equivariant $O(2)$-Hopf bifurcation can support invariant torus which is formed by oscillation waves, see in particular \cite{Yao, LY}. In our study here, we also get an invariant set which is of bell-shape. However, this invariant bell consists of stationary oscillation waves.

We shall also emphasize that we get the dynamics on the center manifold to arbitrary orders allowed by the smoothness of the flux functions. Hence we get, by using symmetry, that the full dynamics on the center manifold in the sense of the unique finite order approximation property of the center manifold reduction function(s). From the dynamical behaviors of the system on center manifold, we could observe that the third order term in the flux also plays an important role for the current study. This is in sharp contrast to our former study on the equivariant Hopf bifurcation case in which the third order term only enters the angular equations hence does not influence the stability. Besides, the computations in the current study are different from our former methods in which normal form theory are involved.  As $\sigma''$ is not involved in the definition of an admissible critical configuration point, a direct consequence of the above observation is that genuine nonlinearity of the corresponding first order system in \eqref{system} is not needed at all to support the current dynamics. More precisely, system \eqref{system} still supports the current stationary equivariant bifurcation dynamics even if $\sigma''(0)=0$. Of course, we need $\sigma'''(0)\not=0$ to avoid degeneracy. As we can also allow $\sigma'(0)=0$ (see Defintion \ref{def} or
Section 3 below), it is fair to see that the current bifurcation mechanism is not driven by the first order hyperbolic system but by the interplay between the diffusions and nonlinearity in the flux function. This is another difference of the current result with our former results on equivariant Hopf bifurcation for the same system \eqref{system}.  Surprisingly, $\sigma''(0)$ also enters the bifurcation dynamics on the center manifold(s) through square. In the setting of hyperbolic conservation laws, we know that the convex flux and concave flux induce different behaviors of wave phenomena and their stability property. However, here the sign of $\sigma'(0)$ if it is non-vanishing (or equivalently the local convexity or concavity of the flux function around the reference uniform solution), is not essential.

\subsection{Discussion and future work} The current work and the former work of the third author \cite{Yao} aim to
study the equivariant dynamics driven by partial differential equations which exhibit symmetry from dynamical system point of view and using techniques in partial differential equations. As most physical systems exhibit symmetry, we shall further explore the role of the symmetry in the study of dynamics of partial differential equations. We shall also study similar models in higher spatial dimensions, especially the most interesting two and three dimensional ones. A direct obstruction for carrying out similar programs is that the dimension of the center spaces are much larger due to higher spatial dimensions and symmetry.

In \cite{PYZ}, one of us with his coauthors studied the $O(2)$-Hopf bifurcation of viscous shock waves in a channel for some physical systems such as 
compressible Navier-Stokes equations, magnetohydrodynamics (MHD) and viscoelasticity models. A natural next step would be the study of the above systems (Navier-Stokes, MHD, viscoelasticity) and related models in both the two dimensional duct geometry and the most interesting three dimensional geometries under the same spectral scenario as in the current work, or even other spectral scenarios. Our current study sheds light on these further steps since a big advantage of our current study is that everything can be calculated explicitly through Fourier series.

After the completion of the current work, the second and third authors found another elegant way to get the current equivariant bifurcation dynamics through Lyapunov-Schmidt reduction and normal form arguments. And actually, the ``spectum+normal form+dynamical decomposition" argument can apply to various dynamical systems driven by partial differential equations with symmetries.  Related work will be reported elsewhere.


We refer the reader to the works of Bressan \cite{Br1, Br2}, Carr \cite{Ca}, Chicone \cite{Ch}, Golubitsky-Stewart-Schaeffer \cite{GSS}, Haragus and Iooss \cite{HI}, Henry \cite{He}, Iooss and Adelmeyer \cite{IA}, Jia and Sverak \cite{JS1, JS2}, Ma and Wang \cite{MW1, MW2}, Nakanishi and Schlag \cite{NS} and the references therein for similar and related studies in a variety of settings.

\textbf{Convention.} We will use ``$*$" to denote
``complex conjugate", i.e., for $z\in \mathbb{C}$, $z^*$ means the
complex conjugate of $z$; ``$\int f \, dx$" means
``$\int_{-\pi}^{\pi}f\,dx$" for the integrand $f$; for two
nonnegative quantities, ``$A\lesssim B$" means ``$A\leq CB$" for
some constant $C>0$. For $U=\begin{pmatrix}u_1\\u_2\end{pmatrix},
V=\begin{pmatrix}v_1\\v_2 \end{pmatrix}\in \mathbb{C}^2$, ``$\langle
U, V \rangle" \,\mbox{means}\, ``\int_{-\pi}^{\pi} u_1 v_1^* + u_2 v_2^*\,dx$".
We also adopt the standard big $O$ ``$\mathcal{O}$" and small o ``$o$"
notations for limiting processes. For a vector $V$, we use $v_j$ or $v^{(j)}$ to represent its components.
We use ``$[\cdot\,\, , \cdot ]$" to denote commutator: $[F, G]=FG-GF$ for F, G being functions, symbols or operators.
For a linear operator $\mathcal{L}: \mathcal X\mapsto \mathcal X$ on some Banach space $\mathcal X$, we use $\sigma(\mathcal{L})$, $\rho(\mathcal{L})$ to denote
its spectrum set and resolvent set. Further, $\sigma(\mathcal{L})=\sigma_c(\mathcal{L})\cup\sigma_s(\mathcal{L})\cup\sigma_u(\mathcal{L})$, i.e.,
$\sigma(\mathcal{L})$ is the union of the center spectrum set $\sigma_c(\mathcal{L})$, the stable spectrum set $\sigma_s(\mathcal{L})$ and the unstable
spectrum set $\sigma_u(\mathcal{L})$. The associated space decomposition is $\mathcal X=\mathcal X_c\cup \mathcal X_s \cup \mathcal X_u$ where
$\mathcal X_c:=\Big(\frac{1}{2\pi i}\oint_{\gamma_c}(\lambda-\mathcal{L})^{-1}\,d\lambda\Big) \mathcal X$ is the center space, 
$\mathcal X_s:=\Big(\frac{1}{2\pi i}\oint_{\gamma_s}(\lambda-\mathcal{L})^{-1}\,d\lambda\Big) \mathcal X$ is the stable space, 
$\mathcal X_u:=\Big(\frac{1}{2\pi i}\oint_{\gamma_u}(\lambda-\mathcal{L})^{-1}\,d\lambda\Big) \mathcal X$ is the unstable space;
The hyperbolic space is given $\mathcal X_h:=\mathcal X_s \cup \mathcal X_u$. The $\gamma_j$ above is any closed simple curve in the complex plane containing only spectrum of type $j=c, s, u$. 

\section{Functional Analytic Setting}

\setcounter{equation}{0}
\setcounter{theorem}{0}

Now we start our analysis by settling down the functional analytic framework. 

First, we write the system \eqref{system} in the form of nonlinear perturbation system. For this purpose, we need to regard $(\tau, u)$ as 
the perturbation variable around the state $(0, 0)$ and write \eqref{system} as the following system
\begin{equation}\label{perturb}
\begin{cases}
\partial_t \tau-\partial_x u=-a\partial^4_x\tau\\
\partial_t u-\sigma'(0)\partial_x\tau=-\delta\partial^2_x u-\partial^4_x u
+\partial_x\Big( \frac{\sigma''(0)}{2}\tau^2 +  \frac{\sigma'''(0)}{6}\tau^3 +\Gamma(\tau)\Big)
\end{cases}
\end{equation}
where $\Gamma(\tau):=\sigma(\tau)-\sigma(0)- \frac{\sigma''(0)}{2}\tau^2 - \frac{\sigma'''(0)}{6}\tau^3$ and $\Gamma(\tau)=\mathcal O(|\tau|^4)$ when $|\tau|$ is small. 
We shall keep in mind that the system \eqref{perturb} is equivalent to the system \eqref{system}.

Second, we write the nonlinear perturbation system \eqref{perturb} into operator equation form.  For this purpose, we denote
$U=\begin{pmatrix} \tau \\ u\end{pmatrix}$ and $\mathcal L, \mathcal N$ as follows: 
\begin{equation}
\mathcal{L}:=\begin{pmatrix}-a\partial_x^4 & \partial_x \\
\sigma'(0)\partial_x & -\delta\partial_x^2-\partial^4_x
\end{pmatrix}
\end{equation} 
and 
\begin{equation}
\mathcal{N}\begin{pmatrix} \tau \\ u
\end{pmatrix}:=\begin{pmatrix} 0\\ \partial_x\Big( \frac{\sigma''(0)}{2}\tau^2 +  \frac{\sigma'''(0)}{6}\tau^3 +\Gamma(\tau)\Big)
\end{pmatrix}.
\end{equation}
With the above notations, the nonlinear perturbation system can be written symbolically as
\begin{equation}
\partial_t U=\mathcal{L}U+ \mathcal{N}(U).
\end{equation}

In order to emphasize the linear operator and nonlinear term, we may also write the system \eqref{system} as

\begin{equation}
\partial_t\begin{pmatrix} \tau \\ u
\end{pmatrix}=\begin{pmatrix}-a\partial_x^4 & \partial_x \\
\sigma'(0)\partial_x & -\delta\partial_x^2-\partial^4_x
\end{pmatrix}\begin{pmatrix} \tau \\ u
\end{pmatrix}+\mathcal{N}\begin{pmatrix} \tau \\ u
\end{pmatrix}.
\end{equation}

Third, we need to decide the space triplet under which we work and achieve our specific goals. The choice of working space triplet is nontrivial (See also Section \ref{last}). Here $\mathcal L$ is a fourth order linear differential operator on the periodic domain $\mathbb{R}^1/[-\pi, \pi]$. We may consider it as a linear operator on the space $L^2_{per}(-\pi, \pi)$ with domain $H^4_{per}(-\pi, \pi)$. Meanwhile, if we seek solutions  $(\tau, u)$ on the periodic 
Sobolev spaces, the quantities $\int \tau\,dx$ and $\int u\,dx$ are conserved due to the conservative form of the original system \eqref{system} as follows:
\begin{equation}
\partial_t\begin{pmatrix} \tau \\ u
\end{pmatrix}=\partial_x\begin{pmatrix}u-a\partial_x^3  \\
\sigma(\tau)-\delta \partial_x u-\partial_x^3 u
\end{pmatrix}.
\end{equation}

Taking into account of the above considerations and the requirement of center manifold theory, we will first  (but also see Section \ref{last}) work on the space triplet $Z\subset Y\subset X$ given below:
$$X:=\{U\in L^2_{per}(-\pi, \pi); \int_{-\pi}^{\pi}U\,dx=0\},$$
$$Y=X=\{U\in L^2_{per}(-\pi, \pi); \int_{-\pi}^{\pi}U\,dx=0\},$$
$$Z:=\{U\in H^4_{per}(-\pi, \pi); \int_{-\pi}^{\pi}U\,dx=0\}.$$
It is important to notice that the mean zero restriction which comes naturally from the conservative structure of system \eqref{system} also has influence on the spectra of the linear operator $\mathcal L(a_c,\delta_c)$.

\section{Spectral Analysis I}

\setcounter{equation}{0}
\setcounter{theorem}{0}

With the functional analytic preparations above, we are now ready to study the spectra of the linear operator $\mathcal L$ and give explanations of the admissible configuration set associated with an arbitrarily fixed positive interger $k_0$. Based on our choice of space triplet, we shall regard
\begin{equation}\mathcal{L}=\mathcal L(a, \delta):=\begin{pmatrix}-a\partial_x^4 & \partial_x \\
\sigma'(0)\partial_x & -\delta\partial_x^2-\partial^4_x
\end{pmatrix}
\end{equation} 
as a linear operator on the space $X$ with domain $Z$ to study its spectra. To this end, we can proceed by Fourier analysis as we are working on periodic domains.
After Fourier transformation, the differential operator is represented by
$$M_k=\begin{pmatrix}-ak^4 & ik \\
\sigma'(0)k i & \delta k^2-k^4
\end{pmatrix},\,\, k\in\mathbb{Z}, \, k\not=0.$$
Therefore, we have $$\sigma(\mathcal{L})=\cup_{k\in\mathbb{Z}, \, k\not=0}\sigma(M_k).$$
The mode $k=0$ is not included in the above union because we have mean 0 restriction in the definition of $X$. The eigenvalues $\lambda$ of $M_k$ for $k\not =0$ are given by 
\begin{align*}
\det (\lambda-M_k)&=\det \begin{pmatrix}\lambda+a k^4 &- ik \\
-\sigma'(0)k i & \lambda- \delta k^2+k^4
\end{pmatrix}\\
& =(\lambda+a k^4)(\lambda-\delta k^2+k^4)+\sigma'(0)k^2\\
&=\lambda^2+\Big( (a+1)k^4-\delta k^2 \Big)\lambda+ a k^4(k^4-\delta k^2)+\sigma'(0)k^2\\
&=0.
\end{align*}

Before further analysis, we first notice that in the formula for the eigenvalues of $M_k$, i.e.,
\begin{equation}\label{eigen}
\lambda^2+\Big( (a+1)k^4-\delta k^2 \Big)\lambda+ a k^4(k^4-\delta k^2)+\sigma'(0)k^2=0,
\end{equation}
$k$ enters the equation through $k^2$, which in particular implies $\sigma(M_k)=\sigma(M_{-k})$. There is no surprise here as
this is a direct consequence of the $O(2)$-symmetry exhibited by the system \eqref{system} (see Section \ref{last}).

Now we explain the admissible critical  configuration sets $\mathcal A(k_0)$ associated with a nonzero (positive) integer $k_0$ in this paper. As we have studied the equivariant Hopf bifurcation spectral scenario in \cite{Yao, LY}, we study here the spectral scenario that there is one and only one spectrum curve crosses the imaginary axis in $\mathbbm C^1$ through the origin with the purpose of covering the generic cases.

Now let us fix the nonzero (positive) integer $k_0$. To achieve the above spectrum crossing scenario based on the wave number $k_0$, necessarily we need that $M_{k_0}$ contributes zero spectrum for the
linear operator $\mathcal L$ while all the other $M_k$ for $|k|\not= k_0$ do not. In other words, we require the following necessary conditions $(a)$ and $(b)$  in  \eqref{eigen}:

$(a)$ \textit{for the nonzero integer $k_0$,
$$ a_c k_0^4(k_0^2-\delta_c )+\sigma'(0)=0.$$}

$(b)$ \textit{for any nonzero integer $k$ such that $ |k|\neq |k_0|$, and $\omega\in \mathbbm R^1-\{0\}$,
$$[a_c k^4(k^4-\delta_c k^2)+\sigma'(0)-\omega^2]+i[(a_c+1)k^4-\delta_c k^2]\neq 0.$$}

The above conditions $(a)$ and $(b)$ give the necessary conditions for the spectral crossing scenario. We write condition $(b)$ in the above form for the ease of verifying consistency. As we may expect or guess, it is not sufficient for our bifurcation analysis as we need the crossing to be transverse in order to complete the bifurcation analysis. It turns out that conditions $(a)$, $(b)$ together with the following nondegeneracy condition $(c)$ are sufficient for our purpose:  

$(c)$ \textit{for the above $k_0$, $(a_c+1)k_0^2-\delta_c\neq 0$}.

The reasons that we call condition $(c)$ a nondegeneracy condition lie in the following two observations: (1) the condition $(a_c+1)k_0^2-\delta_c\neq 0$ means that $M_{k_0}$ does not contribute repeated zero spectrum for $\mathcal L$ as $(a+1)k^4-\delta k^2=((a+1)k^2-\delta)k^2$ is the coefficient of the first order term in the spectral variable $\lambda$ in \eqref{eigen}; (2) when we compute the dual kernel of the linear operator $\mathcal L(a_c,\delta_c)$ during the reduction procedure, $(a_c+1)k_0^2-\delta_c$ appears in the denominator, see Remark \ref{rc}.

We may notice that the expression $a_c (2k_0)^4((2k_0)^2-\delta_c )+\sigma'(0)$ also appears as a denominator during the reduction procedure, see in particular \eqref{v_1}. However, we do not need to worry about if $a_c (2k_0)^4((2k_0)^2-\delta_c )+\sigma'(0)$ is nonzero or not. Actually, under conditions $(a)$, $(b)$ and $(c)$, the expression $a_c (2k_0)^4((2k_0)^2-\delta_c )+\sigma'(0)$ is automatically nonzero, which is a direct consequence of the paragraph above \eqref{v_1} (see also Remark \ref{r6}). Therefore, in the definition of $\mathcal A(k_0)$, the following condition $(d)$ is implicitly included:

$(d)$ \textit{$a_c (2k_0)^4((2k_0)^2-\delta_c )+\sigma'(0)\neq0$}.\\

Before we do further spectrum analysis, we need first verify that there exists flux function $\sigma(\tau)$ such that $\mathcal A(k_0)$ is nonempty for some nonzero (positive) integer $k_0$, i.e., $(a)$ $(b)$ and $(c)$ are consistent for some $k_0$ as above, which is enough for us to carry out the remaining parts of the program. Actually, we can easily show that there exist flux functions $\sigma(\tau)$ such that $\mathcal A(k_0)$ is nonempty for any nonzero (positive) integer $k_0$. For this purpose, we claim:\\

$(Consistency)$ \textit{conditions $(a)-(d)$ are consistent} in the sense above.\\

This is easily seen by considering the special case $\sigma'(0)=0$. Then $(a)$ and $(d)$ require that $a_c\neq 0$ and $\delta_c=k_0^2$.
As long as $a_c\neq 0$ and $\delta_c=k_0^2$, $(a)$, $(c)$ and $(d)$ are satisfied. We can then adjust $a_c>0$ to satisfy $(b)$, which is easy. In particular if $k_0=\pm 1$, $(b)$ is satisfied since $(a_c+1)k_0^4-\delta_c k_0^2=a_c\neq 0$. Notice that we have no requirement on the sign of $a_c$ here. Alternatively, we can let $a_c<-1$ to satisfy $(b)$. In particular, the operator $\mathcal L(a_c,\delta_c)$ is not a sectoral operator if $a_c<-1$. 

Another easy way to see the consistency is to choose $\delta_c=0$, $a_c\neq0, -1$ and $a_ck_0^6+\sigma'(0)=0$.

Next we give several remarks based on the above analysis and the whole paper.

\begin{remark}
From the analysis in the above two paragraphs, we see that $\mathcal A(k_0)$ is generically nonempty for any nonzero integer $k_0$ and the conditions $(a)$, $(b)$ and $(c)$ are actually rather mild though we may regard on first seeing that they impose strong constraints on the flux functions or the parameters. 
\end{remark}

\begin{remark}
Without condition $(c)$, we can still show the existence of center manifold reduction but we have some problems in computation as this appears as a denominator.
\end{remark}

\begin{remark}\label{r3}
As long as $a_c\neq 0$, the resolvent estimate in Lemma \ref{plemma} is valid, which will be sufficient to guarantee the existence of center manifold. $a_c\neq0$ as long as $(a_c, \delta_c)\in \mathcal A(k_0)$, which is easily seen from conditions $(a)$ and $(d)$.
\end{remark}

\section{Spectrum Analysis II}\label{spectrum2}

\setcounter{equation}{0}
\setcounter{theorem}{0}

From now on, we let $(a_c, \delta_c)$ be an admissible critical configuration point. Due to the spectral preparations in the last section, we are now on a sound foundation to do the bifurcation analysis. In this section, we will introduce the bifurcation system and make some spectral preparations.

First, we introduce several notations:
\begin{equation}
\nu=(\nu_1, \nu_2):=(a-a_c, \delta-\delta_c),\quad \mu=\mu(\nu):=(\delta_c-k_0^2)\nu_1+a_c\nu_2.
\end{equation}
Technically, we can do the bifurcation analysis for the following three cases: (1) $\delta$ varies around $\delta_c$ with $a=a_c$ being fixed; (2) $a$ varies around $a_c$ with $\delta=\delta_c$ being fixed; (3) $(a, \delta)$ varies around $(a_c, \delta_c)$. Formally, case (3) covers cases (1) and (2) and we will first do the analysis based on case (3). However, there are points to be remarked about cases (1) and (2) from both the mathematical and applicational points of view, see Section \ref{last}.

Second, we isolate the bifurcation parameter $\mu$ to get the bifurcation system. For this purpose, we write system \eqref{perturb} in the following form
\begin{align}\label{bs}
\partial_t U=\mathcal L(a_c,\delta_c)U+(\mathcal L(a, \delta)-\mathcal L(a_c,\delta_c))U+\mathcal N (U),
\end{align}
where we have the obvious identifications:
\begin{equation}
\mathcal L(a_c,\delta_c)=\begin{pmatrix}-a_c\partial_x^4 & \partial_x \\
\sigma'(0)\partial_x & -\delta_c\partial_x^2-\partial_x^4
\end{pmatrix},
\end{equation} 

\begin{equation}
\mathcal L(a, \delta)-\mathcal L(a_c,\delta_c)=\begin{pmatrix}-\nu_1\partial_x^4 & 0 \\
0 & -\nu_2\partial_x^2
\end{pmatrix}
\end{equation}
and 
\begin{equation}
\mathcal N (U)= \begin{pmatrix} 0 \\ \partial_x\Big( \frac{\sigma''(0)}{2}\tau^2 +  \frac{\sigma'''(0)}{6}\tau^3 +\Gamma(\tau)\Big)
\end{pmatrix}
\end{equation}

Now we introduce $R(U,\nu):= R_{11}(U,\nu)+R_{20}(U, U)+R_{30}(U, U, U)+\tilde R (U)$ where
$$R_{11}(U,\nu):=\begin{pmatrix}-\nu_1\partial_x^4  U^{(1)}\\ -\nu_2\partial_x^2 U^{(2)}\end{pmatrix}\,\,,R_{20}(U, V):=\begin{pmatrix}0 \\ \frac{\sigma''(0)}{2}\partial_x(U^{(1)}V^{(1)})
\end{pmatrix},$$ 
$$R_{30}(U, V, W):=\begin{pmatrix}0 \\ \frac{\sigma'''(0)}{6}\partial_x(U^{(1)}V^{(1)}W^{(1)})
\end{pmatrix},\,\, \tilde{R}(U):=\begin{pmatrix}0 \\ \partial_x \Gamma(U^{(1)})
\end{pmatrix},$$
and write the nonlinear perturbation system \eqref{bs} around $\mu=0$ as
\begin{equation}\label{bifurcation equation}
\partial_t U=\mathcal L(a_c,\delta_c)U + R(U, \nu).
\end{equation}
The above equation \eqref{bifurcation equation} is the \textit{bifurcation equation or system}.

Finally, we make some spectral preparations for the use of center manifold theory.

\begin{lemma}\label{cp}
$\sigma_c(\mathcal L(a_c, \delta_c))=\{0\}$.\end{lemma}

\begin{proof}
This is a direct consequence of the definition of the admissible critical configuration set $\mathcal A(k_0)$.
\end{proof}

Concerning the spectrum of $\sigma(\mathcal L(a_c,\delta_c))$, we also have the following lemma:

\begin{lemma}\label{gap}
There exists a positive constant $\gamma>0$, such that $\sup\{Re\, \lambda; \lambda\in \sigma_s(\mathcal L(a_c,\delta_c))\}<-\gamma$ and  $\inf\{Re\, \lambda; \lambda\in \sigma_u(\mathcal L(a_c,\delta_c))\}>\gamma$.
\end{lemma}

\begin{proof}
We need to consider the distribution of the roots of equation \eqref{eigen} for $k\neq k_0$. By symmetry, we just need consider the case $|k_0|\neq k\in\mathbb N$. 
By Lemma \ref{cp}, we know that roots of equation \eqref{eigen} for $k\neq k_0$ do not lie in the imaginary axis. Hence we just need to show that there is no accumulation of spectra to the imaginary axis when $k\rightarrow +\infty$. This is obvious by writing down the solutions explicitly through quadratic formula: the real parts of the roots of  \eqref{eigen} can only tend to $\pm\infty$. 
\end{proof}

\begin{remark}\label{con}
$\mathcal \cup_{0\neq k_0\in \mathbbm Z}\mathcal A(k_0)$ contains all the admissible critical configuration points. Our bifurcation analysis will be done in particular around $\nu=(\nu_1,\nu_2)=(0,0)$. The role of the parameter $\mu$ will be self-evident after we get the reduced dynamics (see \eqref{bp}).
\end{remark}

\section{Existence of Parameter-Dependent Center Manifold}\label{STD}

\setcounter{equation}{0}
\setcounter{theorem}{0}

In this section, we show the existence of parameter-dependent center manifold for the bifurcation system \eqref{bifurcation equation} or equivalently the system \eqref{bs} or \eqref{system}. The main ingredient remaining to show is a resolvent estimate. A similar estimate was first done in \cite{Yao}. However, we will reproduce it here to see the roles of the defining conditions in $\mathcal A(k_0)$ and
correct some misprints in \cite{Yao} and also for completeness. Meanwhile, we will add symmetries into consideration in order to make comparisons with our former works (see Section \ref{last}).


Let us first state a version of the parameter-dependent center manifold theorem with group actions.

\begin{theorem} (Parameter dependent center manifold theorem with symmetries, see \cite{HI, GSS, Yao})\label{center manifold theorem}
Let the inclusions in the Banach space triplet $\mathcal Z \subset \mathcal Y\subset \mathcal X$ be continuous. Consider a differential equation in a Banach space $\mathcal X$ of the form
$$\frac{du}{dt}=\mathcal{L}u+\mathcal{R}(u,\nu)$$ and assume that

(1) (Assumption on linear operator and nonlinearity) $\mathcal{L}:\mathcal Z\mapsto\mathcal X$ is a bounded linear map
and for some $k\geq 2$, there exist neighborhoods
$\mathscr{V_u}\subset\mathcal Z$ and
$\mathscr{V_\nu}\subset\mathbb{R}^m$ of $(0,0)$ such that
$\mathcal{R}\in C^k(\mathscr{V}_u\times \mathscr{V_{\nu}},
\mathcal Y)$ and
$$\mathcal{R}(0,0)=0,\,\, D_u\mathcal{R}(0,0)=0.$$

(2) (Spectral decomposition) there exists some constant $\gamma>0$ such that
$$\inf\{Re\lambda; \lambda\in \sigma_u(\mathcal L)\}>\gamma,\,\, \sup\{Re\lambda; \lambda\in \sigma_s(\mathcal L)\}<-\gamma,$$
and the set $\sigma_c(\mathcal L)$ consists of a finite number of eigenvalues with finite algebraic multiplicities.

(3) (Resolvent estimates) For Hilbert space triplet $\mathcal Z \subset \mathcal Y\subset \mathcal X$, assume
there exists a positive constant $\omega_0>0$ such that $i\omega\in \rho(\mathcal L)$ for all $|\omega|>\omega_0$ and
$\|(i\omega-\mathcal L)^{-1}\|_{\mathcal X\mapsto \mathcal X}\lesssim\frac{1}{|\omega|}$. For Banach space triplet, we need further
$\|(i\omega-\mathcal L)^{-1}\|_{\mathcal Y\mapsto \mathcal X}\lesssim\frac{1}{|\omega|^{\alpha}}$ for some $\alpha\in [0,1)$.

Then there exists a map $\Psi\in\mathcal C^k(Z_c, Z_h)$ and a neighborhood 
$\mathscr O_{u}\times \mathscr O_{\nu}$of $(0, 0)$ in $\mathcal Z\times \mathbb R^m$ such that

(a) (Tangency) $\Psi(0, 0)=0$ and $D_u\Psi(0,0)=0$.

(b) (Local flow invariance) the manifold $\mathcal M_0 (\nu)=\{ u_0+\Psi(u_0,\nu); u_0\in Z_c \}$ has the properties.
(i) $\mathcal M_0 (\nu)$ is locally invariant, i.e., if $u$ is a solution satisfying $u(0)\in \mathcal M_0 (\nu)\cap \mathscr O_{\nu}$ and $u(t)\in\mathcal O_{u}$ for all $t\in [0, T]$, then
$u(t)\in\mathcal M_0 (\nu)$ for all $t\in [0, T]$; (ii) $\mathcal M_0 (\nu)$ contains the set of bounded solutions staying in $\mathcal O_{u}$ for all $t\in\mathbb{R}^1$, i.e., if $u$ is a solution satisfying
$u(t)\in\mathcal O_u$ for all $t\in\mathbb R^1$, then $u(0)\in \mathcal M_0 (\nu)$.

(c) (Symmetry) Moreover, if the vector field is equivariant in the sense that there
exists an isometry $\mathscr{T}\in \mathcal L(\mathcal X)\cap \mathcal L(\mathcal Z)$
which commutes with the vector field in the original system,
$$[\mathscr{T}, \mathcal{L}]=0,\,\,[\mathscr{T}, \mathcal{R}]=0,$$
then $\Psi$ commutes
with $\mathscr{T}$ on $Z_c$: $[\Psi, T]=0$.
\end{theorem}

It is easy to see that in the above Theorem \ref{center manifold theorem}, the linear operator is not required to be sectorial. Another subtle point is that the center manifolds are generally not unique but the center manifold reduction function can be approximated uniquely up to any finite order (hence we use ``manifold" and ``manifolds" interchangeably) as long as the nonlinearity has enough smoothness (see for example, Theorem 2.5 in page 35 and Theorem 10 in page 120 of \cite{Ca}).  More insight can be drawn here. In particular, this uniqueness of approximation to any finite order enables us to get full dynamics on center manifold through center manifold reduction. Of course, the distinct properties of the systems under consideration should be taken into account.
We have the following lemma regarding system \eqref{bifurcation equation}. 
\begin{lemma}\label{plemma}
(Existence of parameter-dependent center manifolds)
For system \eqref{bifurcation equation}, there exists a map
$\Psi\in\mathcal{C}^k( Z_c\times \mathbb{R}^1, Z_h)$, with
$$\Psi(0, 0)=0,\,\, D_U\psi(0,0)=0,$$
and a neighborhood of $\mathcal{O}_U\times\mathcal{O}_{\nu}$ of $(0,
0)$ such that for $\nu\in \mathcal{O}{\mu}$, the manifold
$$\mathcal{M}_0(\nu):=\{U_0 +\Psi(U_0, \nu);\,\, U_0\in Z_c\}$$
is locally invariant and contains the set of bounded solutions of
the nonlinear perturbation system in $\mathcal{O}_U$ for all
$t\in\mathbb{R}$. 
\end{lemma}

\begin{proof}
From the definition of the operator $\mathcal L(a_c,\delta_c)$ and $R(U,\nu)$ in \eqref{bifurcation equation}, we know that the 
assumptions (1) and (2) in Theorem \ref{center manifold theorem} on the linear operator and nonlinearity hold. 
\textit{A subtle point here is that the highest 
order of derivatives in $R(U,\nu)$ is four,}  which is allowed by our space triplet choice $Z\subset Y\subset X$ in Section 2 and Theorem \ref{center manifold theorem}.
The spectrum decomposition assumption is a direction consequence of the analysis in Section \ref{spectrum2}, 
see Lemma \ref{cp} and Lemma \ref{gap}. It is obvious that $i\omega\in\rho(\mathcal L(a_c,\delta_c))$ for $|\omega|>0$. 
To show the resolvent estimate in the current Hilbert space triplet setting, we write $(i\omega-\mathcal L(a_c,\delta_c)) U=\tilde U$ for $\tilde U=\begin{pmatrix}\tilde\tau\\ \tilde u \end{pmatrix} \in X$ and 
$U=\begin{pmatrix}\tau \\ u \end{pmatrix} \in Z $ and show that $\|U\|_{X}\lesssim\frac{1}{|\omega|}\|\tilde U\|_{X}$ for $|\omega|>\omega_0>0$ where $\omega_0\gg 1$ is a large constant. Without loss of generality, we just need to prove for $\omega\geq \omega_0$. Multiplying the equation $(i\omega-\mathcal L(a_c,\delta_c)) U=\tilde U$ by $U^*$, integrating over $[-\pi, +\pi]$ and integrating by parts, we arrive at
$$i\omega|\tau|^2_{L^2}+a_c|\partial_x^2\tau|_{L^2}^2-\int \partial_x u\tau^*=\int \tilde{\tau}\tau^*;$$
$$\int-\sigma'(0)\partial_x\tau u^*+i \omega |u|_{L^2}^2-\delta_c\int|\partial_x u|^2+\int|\partial_x^2 u|^2=\int\tilde{u}u^*.$$
Taking the imaginary and real parts respectively, we see

$$\omega|\tau|_{L^2}^2=Im\int \partial_x u\tau^* +Im \int \tilde{\tau}\tau^*,$$
$$\omega |u|_{L^2}^2=  \sigma'(0) Im  \int \partial_x\tau u^* +Im \int \tilde{u}u^*$$
and
$$a_c |\partial_x^2\tau|^2_{L^2}=Re\int \partial_x u\tau^* +Re\int \tilde{\tau}\tau^*,$$
$$-\delta_c |\partial_x u|^2_{L^2}+\int|\partial_x^2 u|^2=Re\int\tilde{u}u^* +Re\int \sigma'(0)\partial_x\tau u^*.$$

From the imaginary part equations, we get by using elementary
inequalities, Fourier analysis and mean zero property for elements in the space $X$, that
\begin{equation}\label{r}
\omega |\tau|^2_{L^2}\lesssim\frac{1}{\omega}\Big(|\partial_x u|^2_{L^2}+|\tilde{\tau}|^2_{L^2}  \Big)\leq \frac{1}{\omega}\Big(|\partial^2_x u|^2_{L^2}+|\tilde{\tau}|^2_{L^2}  \Big),
\end{equation}
\begin{equation}\label{i}
\omega |u|_{L^2}^2\lesssim \frac{1}{\omega}\Big(|\partial_x\tau|^2_{L^2}+|\tilde{u}|^2_{L^2}  \Big)\leq  \frac{1}{\omega}\Big(|\partial^2_x\tau|^2_{L^2}+|\tilde{u}|^2_{L^2}  \Big).
\end{equation}
In view of $a_c\neq 0$ (see Remark \eqref{r3}), we obtain from the real part equations that
\begin{equation}\label{real1}
|\partial_x^2\tau|^2_{L^2}\lesssim|\tau|_{L^2}\Big( |\tilde{\tau}|_{L^2}+|\partial_x u|_{L^2}  \Big),\end{equation}
\begin{equation}\label{real2}
|\partial_x^2 u|^2_{L^2}\lesssim|u|_{L^2}\Big(
|\tilde{u}|_{L^2}+|\partial_x\tau|_{L^2}  \Big)+|\partial_x
u|^2_{L^2}.
\end{equation}
Notice that \eqref{real2} is valid for any real numbers $\delta_c$. In particular, $\delta_c$ can be 0. 
By interpolation, we know that for any $\epsilon>0$, the following holds
$$|\partial_x u|^2_{L^2}\leq \epsilon|\partial_x^2 u|^2_{L^2}+C(\epsilon)|u|^2_{L^2}.$$
We may pick $\epsilon$ so small that we can conclude from $\eqref{real2}$
that
$$|\partial_x^2 u|^2_{L^2}\lesssim|u|_{L^2}\Big(|\tilde{u}|_{L^2}+|\partial_x\tau|_{L^2} +|u|_{L^2}  \Big).$$
By interpolation in $|\partial_x\tau|_{L^2}$,
$$|\partial_x\tau|^2_{L^2}\lesssim\epsilon |\partial_x^2\tau|^2_{L^2}+C(\epsilon)|\tau|^2_{L^2},$$
we can get from \eqref{real1} and \eqref{real2} that
\begin{equation}\label{3}
\begin{cases}
|\partial_x^2\tau|^2_{L^2}\lesssim
|\tau|^2_{L^2}+|\tilde{\tau}|^2_{L^2}+\epsilon|\partial_x^2 u|^2_{L^2}+|u|^2_{L^2},\\
|\partial_x^2 u|^2_{L^2}\lesssim
|u|^2_{L^2}+|\tilde{u}|^2_{L^2}+\epsilon|\partial_x^2\tau|^2_{L^2}+\epsilon|\partial_x^2
u|^2_{L^2}.
\end{cases}
\end{equation}
Adding the two equations in \eqref{3} together and choosing $\epsilon$
smaller if necessary, we get
\begin{equation}\label{4}
|\partial_x^2\tau|^2_{L^2}+|\partial_x^2u|^2_{L^2}\lesssim|\tau|^2_{L^2}+|u|^2_{L^2}+|\tilde{\tau}|^2_{L^2}+|\tilde{u}|^2_{L^2}.
\end{equation}
Now we have in view of \eqref{r}, \eqref{i} and \eqref{4} that
\begin{align*}
\omega|\tau|^2_{L^2}&\lesssim\frac{1}{\omega}|\tilde{\tau}|^2_{L^2}+\frac{1}{\omega}|\partial_x^2
u|^2_{L^2}\\
&\lesssim\frac{1}{\omega}|\tilde{\tau}|^2_{L^2}+\frac{1}{\omega}\Big(|\tau|^2_{L^2}+|u|^2_{L^2}+|\tilde{\tau}|^2_{L^2}+|\tilde{u}|^2_{L^2}
\Big).
\end{align*}
\begin{align*}
\omega |u|^2_{L^2}&\lesssim\frac{1}{\omega}|\tilde{u}|^2_{L^2}+\frac{1}{\omega}|\partial_x^2
\tau|^2_{L^2}\\
&\lesssim\frac{1}{\omega}|\tilde{u}|^2_{L^2}+\frac{1}{\omega}\Big(|\tau|^2_{L^2}+|u|^2_{L^2}+|\tilde{\tau}|^2_{L^2}+|\tilde{u}|^2_{L^2}
\Big).
\end{align*}
Adding the above two inequalities together, we get
$$(-\frac{1}{\omega}+\omega) |U|^2_{L^2}\lesssim \frac{1}{\omega} |\tilde U|^2_{L^2},$$
which implies
$$|U|^2_{L^2}\lesssim\frac{1}{\omega(\omega-\frac{1}{\omega})}|\tilde U|^2_{L^2}
=\frac{1}{\omega^2 -1}|\tilde U|^2_{L^2}
\lesssim\frac{1}{\omega^2}|\tilde U|^2_{L^2}.$$
for $\omega\geq \omega_0$ large. Hence we arrive at the estimate of desired
form $|U|_{L^2}\lesssim\frac{1}{\tilde{\omega}}|\tilde U|_{L^2}$. 
\end{proof}

\begin{remark}\label{gcm} The second inequalities 
in \eqref{r} and \eqref{i} hold with our choice of space triplet as there are derivatives acting on the function $\tau$ and $u$.
\end{remark}

Now we define the following operations
\begin{equation}\label{symmetry}
R_{\phi}\begin{pmatrix}\tau(x)\\u(x) \end{pmatrix}=\begin{pmatrix}\tau(x+\phi)\\ u(x+\phi) \end{pmatrix},\quad
S\begin{pmatrix}\tau(x)\\u(x) \end{pmatrix}=\begin{pmatrix}\tau(-x)\\ -u(-x) \end{pmatrix},
\end{equation}
for any $\phi\in\mathbb R/2\pi\mathbb Z$. We can easily verify that system \eqref{system} (equivalently \eqref{bs}) is equivariant:
\begin{equation}\label{equivairance}
[R_{\phi}, \mathcal L(a_c, \delta_c)]=0, [R_{\phi}, R(\cdot, \nu)]=0, [S, R(\cdot, \nu)]=0, [S, \mathcal L(a_c, \delta_c)]=0,
R_{\phi}S=SR_{-\phi}.
\end{equation}
Therefore, the center manifold function in Lemma \ref{plemma} also inherits the above symmetries.

\section{Dynamics on Center Manifold}
\setcounter{equation}{0}
\setcounter{theorem}{0}

In this section, we compute and analyze the dynamics on the center manifold for our bifurcation system \eqref{bifurcation equation}. As we want to demonstrate some subtle points during the process and also to make the exposition clear, we divide this section into subsections.

\subsection{Center space and parametrization}

First, we compute the center space of the operator $\mathcal L(a_c,\delta_c)$. Recall that $$\mathcal L(a_c, \delta_c):=\begin{pmatrix} -a_c \partial_x^4 & \partial_x\\ \sigma'(0)\partial_x & -\delta_c \partial_x^2-\partial_x^4 \end{pmatrix}.$$

For the wave number $k_0$, letting $\xi=e^{ik_0x}V=e^{ik_0x}\begin{pmatrix}v_1\\ v_2 \end{pmatrix}$, we get in Fourier side the system satisfied by $V$:

\begin{equation}\label{V1}
\begin{pmatrix}-a_c k_0^4 &ik_0\\ \sigma'(0)ik_0 &
\delta_c k_0^2-k_0^4\end{pmatrix}\begin{pmatrix}v_1\\ v_2
\end{pmatrix}=0\begin{pmatrix} v_1\\ v_2
\end{pmatrix},
\end{equation}
which is equivalent to

\begin{equation}
\begin{cases}
-a_c k_0^4v_1+ik_0 v_2=0\\
\sigma'(0)ik_0v_1+(\delta_ck_0^2-k_0^4)v_2=0.
\end{cases}
\end{equation}

From condition $(a)$ in the definition of $\mathcal A(k_0)$, we know that 
$$\det \begin{pmatrix}-a_c k_0^4 &ik_0\\ \sigma'(0)ik_0 &
\delta_c k_0^2-k_0^4\end{pmatrix}=k_0^2[a_ck_0^4(k_0^2-\delta_c)+\sigma'(0)]=0$$ 
and get
$$a_ck_0^4v_1=ik_0 v_2,\quad i.e.,\quad v_2=\frac{a_ck_0^4}{ik_0}v_1=-ia_ck_0^3 v_1.$$

Consequently, we can choose
$$\xi=e^{ik_0x}\begin{pmatrix}1\\-ia_ck_0^3 \end{pmatrix}.$$

By conjugacy, for the wave number $-k_0$, we seek solutions of the form $e^{-ik_0x}V=e^{-ik_0x}\begin{pmatrix}v_1\\ v_2 \end{pmatrix}$ and get that
$$\xi^*=e^{-ik_0x}\begin{pmatrix}1\\ia_ck_0^3 \end{pmatrix}.$$

Then the center space $Z_h$ which is nothing but $\ker \mathcal L(a_c,\delta_c)$ can be parametrized by
$$Z_h=\{A\xi+A^*\xi^*;\, A\in\mathbbm C^1\}.$$
This parametrization will be important for the later computations.

\subsection{The dual Kernel $\mathcal L^*(a_c,\delta_c)$} Next, we compute the kernel of the conjugate operator $\mathcal L^*(a_c,\delta_c)$. By simple checking based on the definition of conjugate operator, we get
$$\mathcal L^*(a_c, \delta_c):=\begin{pmatrix} -a_c \partial_x^4 & -\sigma'(0)\partial_x\\ -\partial_x & -\delta_c \partial_x^2-\partial_x^4 \end{pmatrix}.$$

Seeking elements of the form $\eta=\kappa e^{ik_0x}V=\kappa e^{ik_0x}\begin{pmatrix}v_1\\ v_2 \end{pmatrix}$ in the dual kernel where $\kappa$ is a renomalization constant to be picken, we get
\begin{equation}\label{V1}
\begin{pmatrix}-a_c k_0^4 &\sigma'(0)ik_0\\ -ik_0 &
\delta_c k_0^2-k_0^4\end{pmatrix}\begin{pmatrix}v_1\\ v_2
\end{pmatrix}=0\begin{pmatrix} v_1\\ v_2
\end{pmatrix},
\end{equation}
which is equivalent to
\begin{equation}
\begin{cases}
-a_c k_0^4v_1-\sigma'(0)ik_0 v_2=0\\
-ik_0v_1+(\delta_ck_0^2-k_0^4)v_2=0.
\end{cases}
\end{equation}

Since $\det \begin{pmatrix}-a_c k_0^4 &ik_0\\ \sigma'(0)ik_0 &
\delta_c k_0^2-k_0^4\end{pmatrix}=0$, we get the relation between $v_1$ and $v_2$ as
$$-a_ck_0^4v_1-\sigma'(0)ik_0v_2=0.$$

Since $a_c\neq0$, we know $v_1=\frac{\sigma'(0)ik_0}{-a_ck_0^4}v_2=-i\frac{\sigma'(0)}{a_ck_0^3}v_2=-ik_0(\delta_c-k_0^2)v_2$. Consequently, we get
$$\eta=\kappa  e^{ik_0x}\begin{pmatrix} -ik_0(\delta_c-k_0^2)\\1 \end{pmatrix}.$$

For computational convenience, we choose $\kappa$ such that $\langle \eta, \xi\rangle=1$. By direct computation, we get

\begin{align*}
\langle \eta, \xi\rangle&=\langle \kappa  e^{ik_0x}\begin{pmatrix} -ik_0(\delta_c-k_0^2)\\1 \end{pmatrix}, e^{ik_0x}\begin{pmatrix} 1\\-ia_ck_0^3 \end{pmatrix}  \rangle\\
&=\int \kappa  e^{ik_0x}\begin{pmatrix} -ik_0(\delta_c-k_0^2)\\1 \end{pmatrix} e^{-ik_0x}\begin{pmatrix} 1\\ia_ck_0^3 \end{pmatrix}\\
&=2\pi[-ik_0(\delta_c-k_0^2)+ia_ck_0^3]\kappa\\
&= 2\pi ik_0[(a_c+1)k_0^2-\delta_c]\kappa\\
&=1,
\end{align*}
which suggests that
$$\kappa=\Big(2\pi ik_0[(a_c+1)k_0^2-\delta_c]\Big)^{-1}$$
and
$$ \eta=\Big(2\pi ik_0[(a_c+1)k_0^2-\delta_c]\Big)^{-1}  e^{ik_0x}\begin{pmatrix} -ik_0(\delta_c-k_0^2)\\1 \end{pmatrix}.$$
 
 By our choice of $\kappa$ and conjugacy of inner product, we have the following duality numerical relations:
$$\langle \eta, \xi \rangle=1,  \langle \eta^*, \xi^* \rangle=1, \langle \eta^*, \xi \rangle=0, \langle \eta, \xi^* \rangle=0.$$
Also, we have obtained by now
$$\ker\mathcal L^*(a_c,\delta_c)=\{B\eta+B^*\eta^*;\, B\in\mathbbm C^1\}.$$
We end this subsection by the following remark:
\begin{remark}\label{rc}
 The denominator in $\kappa$ which is the coefficient of the first order term in the characteristic equation \eqref{eigen} up to a nonzero constant  involving $k_0^2$ is not zero for $(a_c,\delta_c)\in \mathcal A(k_0)$.  
 \end{remark}

\subsection{The projection operator $\mathbbm{P}$} In this subsection, we will introduce the projection operator $\mathbbm P$ to make the remaining computations elegant and efficient. To do so, we decompose the space triplet $Z\subset Y\subset X$ as follows:
$$Z=Z_c\oplus Z_h,\, Y=Y_c\oplus Y_h, \, X=X_c\oplus X_h,$$
where the subindices ``$c$" and ``$h$" stand for center and hyperbolic spaces respectively and
$$Z_c=Y_c=X_c=\{A\xi+A^*\xi^*;\, A\in\mathbbm C^1\}.$$
For an element $U\in Z$ or in $Y, X$, we decompose it according to the above space decompositions as
\begin{align*}
U=&A\xi+A^* \xi^*+(U-A\xi-A^* \xi^*)\\
=& \langle \eta^*, U \rangle \xi+ \langle \eta, U \rangle \xi^*+(U- \langle \eta^*, U \rangle \xi- \langle \eta, U \rangle \xi^*)\\
=&\mathbbm PU+(1-\mathbbm P)U
\end{align*} 
where the projection operator $\mathbbm P$ is defined as
$$\mathbbm P\cdot=\langle \eta^*, \cdot \rangle \xi+ \langle \eta, \cdot \rangle \xi^*.$$

\subsection{The center manifold reduction function}
In this subsection, we will compute the second order approximation of the center manifold reduction function $\Psi$ to make preparations for the study of full dynamics of \eqref{system} on center manifolds.
By now, we have shown the existence of center manifold and parametrized the center space of $\mathcal L(a_c,\delta_c)$ in terms of the complex conjugate pair $(A, A^*)$.
Hence we can correspondingly decompose an element $U=U(x,t)$ by $U=A(t)\xi+A(t)\xi^* +\Psi(A, A^*)$ where
$A=A(t)\in \mathbbm{C}^1$ is a complex function in time variable $t$ and $\Psi(A, A^*)\in Y_h$ is the center manifold function.

For briefty, we define
$$R_{20}(A, A^*):=R_{20}(A(t)\xi+A(t)\xi^*, A(t)\xi+A(t)\xi^*)$$
and
$$R_{30}(A, A^*):=R_{20}(A(t)\xi+A(t)\xi^*, A(t)\xi+A(t)\xi^*, A(t)\xi+A(t)\xi^*).$$

The second order approximation of the center manifold function $\Psi(A, A^*)$ is given by $\Psi(A, A^*)=(-\mathcal L(a_c,\delta_c))^{-1}(1-\mathbbm P)R_{20}(A, A^*)+O(|A|^3)$.
Simple computation yields
\begin{align*}
R_{20}(A, A^*)&=\begin{pmatrix}0\\ \frac{\sigma''(0)}{2}\partial_x[(Ae^{ik_0x}+A^*e^{-ik_0x})^2] \end{pmatrix}\\
&= ik_0\sigma''(0)\begin{pmatrix} 0\\ e^{2ik_0x}A^2-e^{-2ik_0x}A^{*2} \end{pmatrix}\\
&=ik_0\sigma''(0)[e^{2ik_0x}\begin{pmatrix} 0\\ A^2\end{pmatrix}-e^{-2ik_0x}\begin{pmatrix} 0\\ A^{*2}\end{pmatrix}]
\end{align*}

We can easily check

\begin{align*}
\langle \eta^*, R_{20}(A, A^*) \rangle&=\int \kappa e^{-ik_0x}\begin{pmatrix}ik_0 (\delta_c-k_0^2)\\1 \end{pmatrix}\cdot[(-ik_0\sigma''(0))[e^{2ik_0x}\begin{pmatrix} 0\\ A^2\end{pmatrix}-e^{-2ik_0x}\begin{pmatrix} 0\\ A^{*2}\end{pmatrix}]\\
&=\int \kappa\begin{pmatrix}ik_0 (\delta_c-k_0^2)\\1 \end{pmatrix}\cdot[(-ik_0\sigma''(0))[e^{ik_0x}\begin{pmatrix} 0\\ A^2\end{pmatrix}-e^{-3ik_0x}\begin{pmatrix} 0\\ A^{*2}\end{pmatrix}]\\
&=0.
\end{align*}

Similarly, we can also check that $\langle \eta, R_{20}(A, A^*) \rangle=0$. Hence we conclude 
$$\mathbbm PR_{20}(A, A^*)=\langle \eta^*, R_{20}(A, A^*) \rangle \xi+ \langle \eta, R_{20}(A, A^*) \rangle \xi^*=0.$$
Consequently, 
$$\Psi(A, A^*)=(-\mathcal L(a_c,\delta_c))^{-1}(1-\mathbbm P)R_{20}(A, A^*)=(-\mathcal L(a_c,\delta_c) )^{-1}R_{20}(A, A^*)+O(|A|^3).$$

To proceed further, we compute $(-\mathcal L(a_c,\delta_c))^{-1}R_{20}(A, A^*)$ which we define by $\Phi(A, A^*)$. As $R_{20}(A,A^*)$ is in the hyperbolic space,
the equality 
$$(-\mathcal L(a_c,\delta_c))^{-1}R_{20}(A, A^*)=\Phi(A, A^*)$$
 is equivalent to that 
$$-\mathcal L(a_c,\delta_c) \Phi(A, A^*)=R_{20}(A,A^*),$$
hence can be written specifically as

\begin{equation}\label{Phi}
-\mathcal L(a_c,\delta_c)\Phi(A, A^*)=ik_0\sigma''(0)[e^{2ik_0x}\begin{pmatrix} 0\\ A^2\end{pmatrix}-e^{-2ik_0x}\begin{pmatrix} 0\\ A^{*2}\end{pmatrix}]
\end{equation}

Due to the form in the right hand side of the above equation and the fact that $\Phi(A, A^*)$ is real-valued, we should assume that
$$\Phi(A,A^*)=i(e^{2ik_0x}V-e^{-2ik_0x}V^*),$$
where $V\in \mathbbm C^2$ is to be determined. The above observation enables us to carry out the current program.
Now we observe that \eqref{Phi} is equivalent to that

\begin{equation*}
-\mathcal L(a_c,\delta_c)e^{2ik_0x}V=k_0\sigma''(0)e^{2ik_0x}\begin{pmatrix} 0\\ A^2\end{pmatrix}.
\end{equation*} 

To solve for the complex vector $V\in\mathbbm C^2$, we write the above equation specifically as

\begin{equation}
\begin{pmatrix} 16a_c k_0^4 & -2ik_0 \\ -2ik_0\sigma'(0) & -4\delta_ck_0^2+16k_0^4 \end{pmatrix}\begin{pmatrix} v_1\\v_2\end{pmatrix}=\begin{pmatrix} 0\\ k_0\sigma''(0)A^2\end{pmatrix}.
\end{equation}

Before we solve this elementary system of algebraic equation, let us first claim that
$$\det\begin{pmatrix} 16a_c k_0^4 & -2ik_0 \\ -2ik_0\sigma'(0) & -4\delta_ck_0^2+16k_0^4 \end{pmatrix}\neq 0.$$
The above claim is a simple consequence of the facts that $k_0\sigma''(0)e^{2ik_0x}\begin{pmatrix} 0\\ A^2\end{pmatrix}\in X_h$, $\Phi(A,A^*)\in X_h$, and
$-\mathcal L(a_c,\delta_c)|_{X_h}$ is invertible. Actually, we can also directly compute it out as 

$$\det\begin{pmatrix} 16a_c k_0^4 & -2ik_0 \\ -2ik_0\sigma'(0) & -4\delta_ck_0^2+16k_0^4 \end{pmatrix}=(2k_0)^2[\sigma'(0)+a_c(2k_0)^4((2k_0)^2-\delta_c)]\neq 0$$
at any admissible critical configuration pairs.
Now simple computations yield

\begin{equation}\label{v_1}
\begin{cases}
v_1=\frac{i\sigma''(0)A^2}{2[\sigma'(0)+a_c(2k_0)^4((2k_0)^2-\delta_c)]}\\
v_2=\frac{8a_c k_0^3 \sigma''(0)A^2}{2[\sigma'(0)+a_c(2k_0)^4((2k_0)^2-\delta_c)]}
\end{cases}
\end{equation}

Noticing the relation $a_ck_0^4(k_0^2-\delta_c)+\sigma'(0)=0$, we get

$$\sigma'(0)+a_c(2k_0)^4((2k_0)^2-\delta_c)=3a_ck_0^4(21k_0^2-5\delta_c)\neq 0.$$

Then we can simplify the expressions of $v_1$ and $v_2$ further as

\begin{equation}
\begin{cases}
v_1=\frac{i\sigma''(0)A^2}{6a_ck_0^4(21k_0^2-5\delta_c)}\\
v_2=\frac{8a_c k_0^3 \sigma''(0)A^2}{6a_ck_0^4(21k_0^2-5\delta_c)}.
\end{cases}
\end{equation}

To avoid ambiguity and for later reference, we will denote from now on that

$$\phi_1=\frac{i\sigma''(0)A^2}{6a_ck_0^4(21k_0^2-5\delta_c)}, \,\, \phi_2=\frac{8a_c k_0^3 \sigma''(0)A^2}{6a_ck_0^4(21k_0^2-5\delta_c)}.$$

To sum up, we have comuted that

$$\Phi(A,A^*)=i(e^{2ik_0x}\begin{pmatrix}\frac{i\sigma''(0)A^2}{6a_ck_0^4(21k_0^2-5\delta_c)}\\ \frac{8a_c k_0^3 \sigma''(0)A^2}{6a_ck_0^4(21k_0^2-5\delta_c)}\end{pmatrix}-e^{-2ik_0x}\begin{pmatrix}\frac{-i\sigma''(0)A^2}{6a_ck_0^4(21k_0^2-5\delta_c)}\\ \frac{8a_c k_0^3 \sigma''(0)A^2}{6a_ck_0^4(21k_0^2-5\delta_c)}\end{pmatrix}),$$
and 
$$\Psi(A,A^*)=i(e^{2ik_0x}\begin{pmatrix}\frac{i\sigma''(0)A^2}{6a_ck_0^4(21k_0^2-5\delta_c)}\\ \frac{8a_c k_0^3 \sigma''(0)A^2}{6a_ck_0^4(21k_0^2-5\delta_c)}\end{pmatrix}-e^{-2ik_0x}\begin{pmatrix}\frac{-i\sigma''(0)A^2}{6a_ck_0^4(21k_0^2-5\delta_c)}\\ \frac{8a_c k_0^3 \sigma''(0)A^2}{6a_ck_0^4(21k_0^2-5\delta_c)}\end{pmatrix})+O(|A|^3).$$

We will again end this subsection with remarks:
\begin{remark}
The form of $\Phi(A, A^*)$ is inherited from the symmetry of the system \eqref{bifurcation equation}.
\end{remark}
\begin{remark}\label{r6}
The claim right before the equation \eqref{v_1} in particular shows that condition $(d)$ is a consequence of conditions $(a)$, $(b)$ and $(c)$ in the definition of $\mathcal A(k_0)$.
\end{remark} 

\subsection{Dynamics on center manifold}
In this subsection, we compute the reduced dynamics for the system \eqref{system} or equivalently the system \eqref{bs} or \eqref{bifurcation equation}.

By our specific parametrization of the center space through complex-conjugate coordinates, the full dynamics of the system \eqref {system} or equivalently the system \eqref{bs} or \eqref{bifurcation equation} on center manifold is given by 
\begin{align*}
&\frac{d}{dt}(A\xi+A^*\xi+\Psi(A,A^*))\\=&\mathcal{L}(a,\delta)(A\xi+A^*\xi+\Psi(A, A^*))+\mathcal N(A\xi+A^*\xi+\Psi(A, A^*))\\
=&\mathcal{L}(a,\delta)(A\xi+A^*\xi^*)+R_{20}(A\xi+A^*\xi^*+\Psi(A, A^*), A\xi+A^*\xi^*+\Psi(A, A^*))\\&+R_{30}(A\xi+A^*\xi^*+\Psi(A, A^*), A\xi+A^*\xi^*+\Psi(A, A^*),A\xi+A^*\xi^*\\&+\Psi(A, A^*))+ \tilde{R}(A\xi+A^*\xi^*+\Psi(A, A^*))\\
\end{align*}

In view of flow invariance  and by direct computations through expanding the left-hand side or directly referring to page 245 of \cite{MW2} (or pages 56-62 of \cite{MW1}), we know that the second order approximation of the dynamics on the center manifold is given by
\begin{equation}
\frac{d}{dt}(A(t)\xi+A(t)\xi^*)=\mathbbm P \mathcal L(a,\delta)(A\xi+A\xi^*+\Phi(A,A^*))+\mathbbm P R_{20}(A\xi+A^*\xi^*+\Phi(A,A^*))+\sum_{j=2}^3\mathbbm P R_{j0}(A, A^*).
\end{equation}

Next, we proceed to compute the dynamics on the center manifold.  During the procedure, we make use of the trivial fact that $\int_{-\pi}^{\pi}e^{ikx}\,dx=0$ for any nonzero integer $k$. Later, we will also see that the second approximation together with the latter fact will give enough insight for the full dynamics on the center manifold(s).

First, we compute $\mathbbm P R_{20}(A,A^*)$. From former analysis, we know that $R_{20}\in Z_h$, which yields $\mathbbm P R_{20}(A,A^*)=0$.

Second, we compute $\mathbbm P R_{20}(A\xi+A^*\xi^*+\Phi(A, A^*), A\xi+A^*\xi^*+\Phi(A, A^*))$. To proceed, we first observe that

$$(A\xi+A^*\xi^*+\Phi(A, A^*))^{(1)}=Ae^{ik_0x}+A^*e^{-ik_0x}+i(e^{2ik_0x}\phi_1-e^{-2ik_0x}\phi_1^{*});$$
$$\partial_x (A\xi+A^*\xi^*+\Phi(A, A^*))^{(1)}=ik_0 (Ae^{ik_0x}-A^*e^{-ik_0x})-2k_0(e^{2ik_0x}\phi_1+e^{-2ik_0x}\phi_1^*).$$

Consequently, we have 

$$(A\xi+A^*\xi^*+\Phi(A, A^*))^{(1)}\partial_x (A\xi+A^*\xi^*+\Phi(A, A^*))^{(1)}=\mathcal Q_1+\mathcal Q_2 +\mathcal Q_3 +\mathcal Q_4$$
where

\begin{align*}\mathcal Q_1&=ik_0(Ae^{ik_0x}+A^*e^{-ik_0x})(Ae^{ik_0x}-A^*e^{-ik_0x})\\&=ik_0(A^2 e^{2ik_0x}-A^{*2}e^{-2ik_0x});\end{align*}

\begin{align*}
\mathcal Q_2&=-2k_0(Ae^{ik_0x}+A^*e^{-ik_0x})(e^{2ik_0x}\phi_1+e^{-2ik_0x}\phi_1^*)\\&=-2k_0(A\phi_1 e^{3ik_0x}+A\phi_1^* e^{-ik_0x}+A^*\phi_1 e^{ik_0x}+A^*\phi_1^*e^{-3ik_xx});
\end{align*}

\begin{align*}
\mathcal Q_3&=-k_0(e^{2ik_0x}\phi_1-e^{-2ik_0x}\phi_1^*)(Ae^{ik_0x}-A^*e^{-ik_0x})\\&=-k_0(A\phi_1 e^{3ik_0x}-A^*\phi_1 e^{ik_0x}-A\phi_1^* e^{-ik_0x}+A^*\phi_1^*e^{-3ik_xx});
\end{align*}

\begin{align*}
\mathcal Q_4&=-2ik_0(e^{2ik_0x}\phi_1-e^{-2ik_0x}\phi_1^*)(e^{2ik_0x}\phi_1+e^{-2ik_0x}\phi_1^*)\\&=-2ik_0(e^{4ik_0x}\phi_1^2-e^{-4ik_0x}\phi_1^{*2}).
\end{align*}

Now we compute a specific form of $R_{20}(A\xi+A^*\xi^*+\Phi(A, A^*), A\xi+A^*\xi^*+\Phi(A, A^*))$ as follows
\begin{align*}
&R_{20}(A\xi+A^*\xi^*+\Phi(A, A^*), A\xi+A^*\xi^*+\Phi(A, A^*))\\&=\frac{\sigma''(0)}{2}\partial_x\begin{pmatrix}0\\ ((A\xi+A^*\xi^*+\Phi(A, A^*)^2 \end{pmatrix}\\
&=\sigma''(0)\begin{pmatrix}0\\ (A\xi+A^*\xi^*+\Phi(A, A^*)^{(1)}\partial_x (A\xi+A^*\xi^*+\Phi(A, A^*))^{(1)}\end{pmatrix}\\
&=\sigma''(0)\begin{pmatrix}0\\ \mathcal Q_1+\mathcal Q_2 +\mathcal Q_3 +\mathcal Q_4\end{pmatrix}.
\end{align*}

Now, it is easy to observe that

$$\langle  \eta^*, \begin{pmatrix}0\\ \mathcal Q_1  \end{pmatrix}\rangle=0,\,\, \langle  \eta^*, \begin{pmatrix}0\\ \mathcal Q_4  \end{pmatrix}\rangle=0;$$
$$\langle  \eta^*, \begin{pmatrix}0\\ \mathcal Q_2  \end{pmatrix}\rangle=\int_{-\pi}^{\pi}(-\kappa e^{-ik_0x})(-2k_0A^*\phi_1 e^{ik_0x})\,dx=4\pi\kappa k_0 A^*\phi_1;$$
$$\langle  \eta^*, \begin{pmatrix}0\\ \mathcal Q_3  \end{pmatrix}\rangle=\int_{-\pi}^{\pi}(-\kappa e^{-ik_0x})(k_0A^*\phi_1 e^{ik_0x})\,dx=-2\pi\kappa k_0 A^*\phi_1.$$
From the expression of $\mathcal Q_j$ for $1\leq j\leq 4$ and the above expression on $R_{20}(A\xi+A^*\xi^*+\Phi(A, A^*), A\xi+A^*\xi^*+\Phi(A, A^*))$, we obtain

\begin{align*}
&\mathbbm P R_{20}(A\xi+A^*\xi^*+\Phi(A, A^*), A\xi+A^*\xi^*+\Phi(A, A^*))\\&=\langle\eta^*, R_{20}(A\xi+A^*\xi^*+\Phi(A, A^*)\rangle \xi+\langle \eta,R_{20}(A\xi+A^*\xi^*+\Phi(A, A^*)\rangle \xi^*\\
&=\sigma''(0)\sum_{j=1}^{4}\langle \eta^*, \begin{pmatrix}0\\ \mathcal Q_j  \end{pmatrix} \rangle\xi+\sigma''(0)\sum_{j=1}^{4}\langle \eta, \begin{pmatrix}0\\ \mathcal Q_j  \end{pmatrix} \rangle\xi^*\\
&=2\pi k_0\sigma''(0)(\kappa A^*\phi_1\xi+\kappa^* A\phi_1^*\xi^*)
\end{align*}

Third, we compute $\mathbbm PR_{30}(A,A^*)$. First, we note that
\begin{align*}
R_{30}(A, A^*)&=\begin{pmatrix}0\\ \frac{\sigma'''(0)}{6}\partial_x((A\xi+A^*\xi^*)^{(1)3})  \end{pmatrix}\\
&=\begin{pmatrix}0\\ \frac{\sigma'''(0)}{2}(A\xi+A^*\xi^*)^{(1)2}\partial_x((A\xi+A^*\xi^*)^{(1)})  \end{pmatrix}\\
&=\begin{pmatrix}0\\ \frac{ik_0\sigma'''(0)}{2}(Ae^{ik_0x}+A^*e^{-ik_0x})(A^2e^{2ik_0x}-A^{*2}e^{-2ik_0x})) \end{pmatrix}\\
&=\begin{pmatrix}0\\ \frac{ik_0\sigma'''(0)}{2}(A^3e^{3ik_0x}-AA^{*2}e^{-ik_0x}+A^2A^*e^{ik_0x}-A^{*3}e^{-3ik_0x}) \end{pmatrix}
\end{align*}

Now, we have

\begin{align*}
\langle \eta^*, R_{30}(A,A^*)\rangle&=\langle \eta^*,  \begin{pmatrix}0\\ \frac{ik_0\sigma'''(0)}{2}(-AA^{*2}e^{-ik_0x})\end{pmatrix}\rangle\\
&=\int_{\pi}^{\pi}(-\kappa e^{ik_0x})(\frac{ik_0\sigma'''(0)}{2}A^2A^*e^{ik_0x})\,dx\\
&=-i\pi\kappa\sigma'''(0)k_0A^2A^*.
\end{align*}

Consequently, we obtain

\begin{align*}
\mathbbm PR_{30}(A,A^*)&=\langle\eta^*,  R_{30}(A,A^*)\rangle \xi+\langle\eta ,R_{30}(A,A^*)\rangle \xi^*\\&=-i\pi\kappa\sigma'''(0)k_0A^2A^*\xi-i\pi\kappa\sigma'''(0)k_0AA^{*2}\xi^*
\end{align*}
where we have used the fact that $i\kappa$ is real.

Fourth, we compute the remaining term 
$$\mathbbm P(\mathcal L(a,\delta)-\mathcal L(a_c,\delta_c))(A\xi+A\xi^*+\Phi(A,A^*)).$$

For this purpose, we recall that
$$\mathcal L(a,\delta)-\mathcal L(a_c,\delta_c)= \begin{pmatrix}-\nu_1\partial_x^4 & 0\\ 0 & -\nu_2\partial_x^2\end{pmatrix}.$$

We note that

\begin{align*}
\partial^4_x(A\xi+A^*\xi^*+\Phi(A,A^*))^{(1)}&=\partial^4_x(Ae^{ik_0x}+A^*e^{-ik_0x}+i(e^{2ik_0x}\phi_1-e^{-2ik_0x}\phi_1^*))\\
&=k_0^4Ae^{ik_0x}+k_0^4 A^* e^{-ik_0x}+i((2k_0)^4e^{2ik_0x}\phi_1-(2k_0)^4e^{-2ik_0x}\phi_1^*)
\end{align*}
and 
\begin{align*}
\partial^2_x(A\xi+A^*\xi^*+\Phi(A,A^*))^{(2)}&=\partial^2_x[Ae^{ik_0x}(-ia_ck_0^3)+A^*e^{-ik_0x}(ia_c k_0^3)+i(e^{2ik_0x}\phi_2-e^{-2ik_0x}\phi_2^*)]\\
&=ia_c k_0^5Ae^{ik_0x}-ia_c k_0^5 A^* e^{-ik_0x}+i[-(2k_0)^2e^{2ik_0x}\phi_2+(2k_0)^2)e^{-2ik_0x}\phi_2^*].
\end{align*}

Hence, we can obtain

\begin{align*}
&\langle \eta^*, (\mathcal L(a,\delta)-\mathcal L(a_c,\delta_c))(A\xi+A\xi^*+\Phi(A,A^*) \rangle\\
&= \langle\eta^*, \begin{pmatrix}-\nu_1\partial^4_x(A\xi+A^*\xi^*+\Phi(A,A^*))^{(1)}\\  -\nu_2\partial^2_x(A\xi+A^*\xi^*+\Phi(A,A^*))^{(2)}\end{pmatrix}\rangle\\
&=\langle\eta^*, e^{-ik_0x}\begin{pmatrix}(-\nu_1) k_0^4A^*\\ (-\nu_2)(-ia_c k_0^5A^*)\end{pmatrix}\rangle\\
&=2\pi(i\kappa)k_0^5[(\delta_c-k_0^2)\nu_1+a_c \nu_2]A,
\end{align*}
which yields
\begin{equation}
\mathbbm P(\mathcal L(a,\delta)-\mathcal L(a_c,\delta_c))(A\xi+A\xi^*+\Phi(A,A^*))=2\pi(i\kappa)k_0^5[(\delta_c-k_0^2)\nu_1+a_c\nu_2](A\xi+A^*\xi^*).
\end{equation}

To sum up, we get the reduced dynamics on the center manifold given by the following equation on $A=A(t)\in \mathbbm C^1$
\begin{align*}
&\frac{d}{dt}A(t)=2\pi(i\kappa)k_0^5[(\delta_c-k_0^2)\nu_1+a_c\nu_2]A(t)+2\pi\sigma''(0)\kappa A^*(t)\phi_1-i\pi\kappa\sigma'''(0)k_0A^2(t)A^*(t)+O(|A|^4)\\
&=\frac{1}{(a_c+1)k_0^2-\delta_c}\Big( k_0^4[(\delta_c-k_0^2)\nu_1+a_c\nu_2] A(t)+[\frac{\sigma''(0)^2}{6a_c k_0^4(21k_0^2-5\delta_c)}-\frac{\sigma'''(0)}{2}]A(t)^2A(t)^* \Big)+O(|A|^4),
\end{align*}
and its complex conjugate equation on $A^*$.
In next subsection, we will strengthen the dynamics to be \eqref{rd} with $g(r)\equiv0$.

\subsection{Analysis of the dynamics on center manifold} \label{analysis}
In this subsection, we shall analyze the full dynamics of \eqref{system} by using the dynamics on center manifolds and therefore prove our main results.
To analyze the dynamics of the system \eqref{bifurcation equation}, it is sufficient to analyze the $A(t)$ equation or equivalently $A^*(t)$ equation. If we introduce polar coordinates $A(t)=r(t)e^{i\theta(t)}$, we get after simple computations that
\begin{equation}\label{pe}
\begin{cases}
\frac{d}{dt}r(t)=\frac{1}{(a_c+1)k_0^2-\delta_c}\Big(k_0^4 [(\delta_c-k_0^2)\nu_1+a_c\nu_2] r(t)+[\frac{\sigma''(0)^2}{6a_ck_0^4(21k_0^2-5\delta_c)}-\frac{\sigma'''(0)}{2}]r(t)^3 \Big)+O(|r|^4)\\
\frac{d}{dt}\theta(t)=O(r^4).
\end{cases}
\end{equation}

When we approximate the center manifold to an arbitrary order for sufficient smooth flux function $\sigma(\tau)$, using the cancellation given by  $\int_{-\pi}^{\pi}e^{ikx}\,dx=0$ for any nonzero integer $k$, we obtain that the higher order terms ($\geq4$) in the center manifold dynamics $\frac{d}{dt}A(t)$ appear as $AP(|A|^2)$ where $P(\cdot)$ is a one variable polynomial with in general complex coefficients and without constant term and first order term. As a demonstration to make this point clear, we can consider typical terms to appear in the $A(t)$ equation, say $\langle \eta, \mathbbm P R_{n0}(A, A^*) \rangle$ where $R_{n0}(A, A^*)=R_{n0}(A\xi+A^*\xi^*, ..., A\xi+A^*\xi^*)$. In view of the definitions of $\eta$, $\xi$ and the inner product, we easily see that $\langle \eta, \mathbbm P R_{n0}(A, A^*) \rangle=0$ for $n$ even and $\langle \eta, \mathbbm P R_{n0}(A, A^*) \rangle=\Omega A^{*k}A^{k+1}=\Omega |A|^{2k}A$ for $n=2k+1$ for $k\geq 2$ for some number $\Omega$. 


Denote $\mathbbm a:=\mathbbm a(k_0, a_c, \delta_c)=\frac{k_0^4}{(a_c+1)k_0^2-\delta_c}$ which is nonzero and real,  and  
$$\mathbbm b:=\mathbbm b(k_0, a_c, \delta_c, \sigma''(0), \sigma'''(0))=\frac{1}{(a_c+1)k_0^2-\delta_c}[\frac{\sigma''(0)^2}{6a_c k_0^4(21k_0^2-5\delta_c)}-\frac{\sigma'''(0)}{2}]
.$$
Now it is natural to choose the bifurcation parameter $\mu$ as
\begin{equation}\label{bp}
\mu=\mu(\nu)=\mu(\nu_1, \nu_2)=(\delta_c-k_0^2)\nu_1+a_c\nu_2.
\end{equation}

Due to the above analysis, we conclude that the full $A(t)$ dynamics on center manifold are actually given by
\begin{equation}\label{fd}
\frac{d}{dt}A(t)= \mathbbm a\mu A+\mathbbm b |A|^2A+O(|A|^5)\\
\end{equation}
in which the $O(|A|^5)$ term is given by $AP(|A|^2)$. If further $\sigma(\tau)$ is assumed to be smooth (which is a technical assumption), then
$P(z)=\sum_{j\geq2}{\mathbbm c_j}z^j$ with $\mathbbm c_j$ being complex in general.

Using subindices $r$ and $i$ to represent real and imaginary parts respectively, the dynamics of equation \eqref{fd} can always be written in polar coordinates as
 \begin{equation}
 \begin{cases}
\frac{d}{dt}r(t)= \mathbbm a_r\mu r+\mathbbm b_r r^3+O(r^5)\\
\frac{d}{dt}\theta(t)=\mathbbm a_i \mu r+\mathbbm b_i r^3+O(r^5)
\end{cases}
\end{equation}
where the $O(r^5)$ term in the $\frac{d}{dt}r(t)$-equation is given by $\sum_{j\geq 2}\mathbbm c_{jr} r^{2j+1}$ while the $O(r^5)$ term in the $\frac{d}{dt}\theta(t)$-equation is given by $\sum_{j\geq 2}\mathbbm c_{ji} r^{2j+1}$ if $\sigma(\tau)$ is smooth.

Noticing that $\mathbbm a$ and $\mathbbm b$ are real, i.e., $\mathbbm a_r=\mathbbm a, \mathbbm a_i=0$ and $\mathbbm b_r=\mathbbm b, \mathbbm b_i=0$, we know that the above system takes the following form 

 \begin{equation}\label{rd}
\begin{cases}
\frac{d}{dt}r(t)= f(r,\mu)\\
\frac{d}{dt}\theta(t)=g(r)
\end{cases}
\end{equation}
where $f(r,\mu):=\mathbbm a\mu r+\mathbbm b r^3+O(|r|^5)$ and $g(r)=\sum_{j\geq 2}\mathbbm c_{ji} r^{2j+1}=O(r^5)$. Notice that $f(r,\mu)$ and $g(r)$ are both real polynomials.

Let us first analyze the radial equation, which is independent of the angular equation. Consider the radial equation $f(r,\mu)=\mathbbm a\mu r+\mathbbm b r^3+O(|r|^5)=0$.  From our former analysis, we can write $f(r,\mu)$ as $f(r,\mu)=rh(r^2, \mu)$ where $h(r,\mu)=\mathbbm a\mu+\mathbbm b r^2+o(|\mu|+r^2)$ with $h$ being at least  a $C^1$-function.  First, we observe that $0$ is always an equilibrium.  Noticing that $h(0,0)=0$ and $\frac{\partial}{\partial (r^2)}h(0,0)=\mathbbm a\neq0$, we conclude by the implicit function theorem that $\mu$ is a function of $r^2$ in a neighborhood of $0$, i.e., $\mu=\tilde g(r^2)$ with $g(0)=0$ for some function $\tilde g$ in a neighborhood of $(0, 0)$. In view that $h(0,0)=0$, we know that the Taylor expansion of $\tilde g$ is $\mu=\tilde g(r^2)=-\frac{\mathbbm a}{\mathbbm b}r^2+o(r^4)$. Then we obtain that there is a curve of nontrivial  equilibria in the $(\mu, r)$-plane that has a second order tangency at $(0, 0)$ to the graph of $\mu=-\frac{\mathbbm b}{\mathbbm a}r^2+o(r^4)$.

Moreover, for the truncated equation $f_{0}(r,\nu):=\mathbbm a\mu r+\mathbbm b r^3=0$, we observe that $r=0$ is always a solution.  Meanwhile, $r=\sqrt{\frac{-\mathbbm a\mu}{\mathbbm b}}>0$, i,e.,
$\mu=-\frac{\mathbbm b}{\mathbbm a}r^2$ is another solution in the parameter range such that $\frac{-\mathbbm a\mu}{\mathbbm b}>0$. In the parameter range $\frac{-\mathbbm a\mu}{\mathbbm b}<0$, $0$ is the only solution. The truncated differential equation $\frac{dr}{dt}=\mathbbm a r+\mathbbm b r^3$ can be solved explicitly: $r^2(t)=\frac{\mathbbm a \mu r_0^2}{a\mu e^{-2\mathbbm a \mu t}+\mathbbm b r_0^2 (e^{-2\mathbbm a\mu t}-1)}$.

The above two paragraphs show that the truncated equation and the full equation have the same number of equilibria in a neighborhood of origin $\mu=0$, which are $o(|\mu|^{1/2})$-close to each other.
Another fact is that the dynamics of the $r$-equation is exactly that of a standard pitchfork bifurcation if we allow $r<0$. As in our case here $r(t)=|A(t)|$, there is only one bifurcated nonzero equilibrium, which is different from the standard pitchfork bifurcation with two bifurcated equilibria. Hence, we may regard the dynamic here as \textit{half pitchfork bifurcation}.

Now, we analyze the angular equation. We need to distinguish two different dynamics: (i) $g(r)=0$, in other words, all the $\mathbbm c_j$ are real and the angular equation is trivial,  (ii) there is some $\mathbbm c_{j_{0}i}\neq0$.
Next we let the nontrivial equilibrium of radial equation be $r_{\mu}$. If the angular equation were nontrivial, then $\theta(t)-\theta_0=g(r_{\mu})t\,\,\mbox{mod}\,\, 2\pi=O(|\mu|^{5/2})$ and due to $SO(2)$-symmery (obviously, $SO(2)$ is a subgroup of $O(2)$), i,e., rotational symmetry, any nontrivial equilibrium $r_{\mu}$ would correspond to a \textit{rotation wave} with radius $r_{\mu}$ and angular speed $O(|\mu|^{5/2})$, which is slower in order compared with amplitude. However, next we will use symmetry to exclude the possibility (ii), i.e., we will show $g(r)\equiv0$. The main idea is that the reduced equation on the center manifold(s) inherit 
the symmetry of the original system by applying (c) of Theorem \ref{center manifold theorem}.

In view of our choice of parametrization of the center space and that $
R_{\phi}\xi=e^{i\phi}\xi, S\xi=\xi^*$, we know on the center space coordinated by $(A, A^*)$, the 
operator $R_{\phi}$, $S$ act as the following $2\times 2$ matrices respectively:
$$\begin{pmatrix} e^{k_0\phi} & 0\\ 0 & e^{-k_0\phi} \end{pmatrix}, \quad \begin{pmatrix} 0 & 1\\ 1 & 0 \end{pmatrix}.$$
Therefore, due to the inherited symmetry, \eqref{fd} must have the form $\frac{d}{dt}A(t)= F(A, A^*, \mu)$ with $F$ satisfying for any $\phi$,
$$F(e^{ik_0\phi}A, e^{-ik_0\phi}A^*, \mu)=e^{ik_0\phi} F(A, A^*, \mu),\quad F(A^*, A, \mu)=F(A, A^*, \mu)^*.$$
By choosing $\phi=-\frac{\arg A}{k_0}$ and then $\phi=\frac{\pi}{k_0}-\frac{\arg A}{k_0}$, we find $F(|A|, |A|, \mu)=e^{-i\arg A}F(A, A^*, \mu)$ and $F(-|A|, -|A|, \mu)=-e^{-i\arg A}F(A, A^*, \mu)$ respectively. Therefore, $F(|A|, |A|, \mu)$ is odd in $|A|$. This implies $F(|A|, |A|,\mu)=|A|G(|A|,\mu)$ for some even function $G(|A|,\mu)$ in its first argument. Therefore, we have $F(A, A^*,\mu)=e^{i\arg A}|A|G(|A|,\mu)=AG(|A|,\mu)$. For our analysis of dynamics on center manifolds, it suffices to consider the case when $F$ is a polynomial. In this case, the analysis above implies $G(|A|, \mu)$ is an even polynomial in $|A|$. Further, the condition $F(A^*, A, \mu)=F(A, A^*, \mu)^*$ forces the coefficient of this even polynomial are real. Hence, we have shown that $g(r)\equiv 0$.

By now, we have shown our main theorem-Theorem \ref{thm}.

\section{Related issues}\label{last}
\setcounter{equation}{0}
\setcounter{theorem}{0}

In this section, we discuss some other subtle points given by the following three subsections.
\subsection{Isolation, Bifurcation Parameters and Measurement} In our bifurcation analysis,  we always isolate the bifurcation parameters in our way of computing dynamics on the center manifold, which is natural. Here we introduced the bifurcation parameter $\mu$ through the bifurcation vector $\nu\in\mathbbm R^2$. Mathematically, the results involving this parameter vector take care of variations of the control parameters $a$ and $\delta$ in one go. Though wide enough, these results are not easy to measure due to the change of two parameters, and also not sharp mathematically in certain cases. We consider the following two cases:

(i) If we fixed $a_c$ and consider the dynamics of $\partial_t U=\mathcal L(a_c,\delta_c)U+(\mathcal L(a,\delta)-\mathcal L(a_c,\delta_c))U+N(U)$ when $\delta$ varies around $\delta_c$. Since $\mathcal L(a,\delta_c)-\mathcal L(a_c,\delta_c)=\begin{pmatrix}0 & 0\\ 0 & -(\delta-\delta_c)\partial_x^2 \end{pmatrix}$,
then it is better for the ease of measurement to introduce the bifurcation parameter $\Gamma_2:=a_c k_0^4(k_0^2-\delta)+\sigma'(0)$. The relation of $\Gamma_2$ and $\nu_2$ is given by $\Gamma_2=-a_c k^4_0 \nu_2$ through the difference $a_c k_0^4 (k_0^2-\delta_c)+\sigma'(0)=0$ and $a_c k_0^4(k_0^2-\delta)+\sigma'(0)=\Gamma_2$. In this case, we have $R(U,\Gamma_2)=\begin{pmatrix}0 & 0\\ 0 & -(\delta-\delta_c) \end{pmatrix}U+N(U)\in H^2_{per}(-\pi, \pi)$. As a consequence, we can choose space triplet $Z\subset Y\subset X$ as $Z=H^4_{per}(-\pi, \pi)$, $Y=H^2_{per}(-\pi, \pi)$ and $X=L^2_{per}(-\pi, \pi)$. Then the corresponding bifurcation occurs in $Y=H^2_{per}(-\pi, \pi)$. This is sharp.

(ii) Similarly, if we fixed $\delta_c$ and consider the dynamics of $\partial_t U=\mathcal L(a_c,\delta_c)U+(\mathcal L(a,\delta)-\mathcal L(a_c,\delta_c))U+N(U)$ when $a$ varies around $a_c$. Since $\mathcal L(a, \delta_c)-\mathcal L(a_c,\delta_c)=\begin{pmatrix}-(a-a_c)\partial_x^4 & 0\\ 0 & 0 \end{pmatrix}$,
then it is better for the ease of measurement to introduce the bifurcation parameter $\Gamma_1:=a k_0^4(k_0^2-\delta_c)+\sigma'(0)$. The relation of $\Gamma_1$ and $\nu_1$ is given by $\Gamma_1= k^4_0(k_0^2-\delta_c) \nu_1$ through the difference $a_c k_0^4 (k_0^2-\delta_c)+\sigma'(0)=0$ and $a k_0^4(k_0^2-\delta_c)+\sigma'(0)=\Gamma_1$. In this case, we have $R(U,\Gamma_2)=\begin{pmatrix}0 & 0\\ 0 & -(\delta-\delta_c) \end{pmatrix}U+N(U)\in L^2_{per}(-\pi, \pi)$. As a consequence, we can not  upgrade the  space triplet choice $Z\subset Y\subset X$ as in (i) above but shall choose $Z=H^4_{per}(-\pi, \pi)$, $Y=X=L^2_{per}(-\pi, \pi)$. Of course, the bifurcation occurs in $Y=L^2_{per}(-\pi, \pi)$ in this case. A drawback is that the observational bifurcation parameter $\Gamma_1$ is degenerate if $k_0^2=\delta_c$ which is allowed.

\subsection{Symmetry and Computations} For our system \eqref{bifurcation equation}, we can verify that it is equivariant under the $O(2)$-group action (hence under $SO(2)$-group action) entirely similarly as in \cite{LY} (see Proposition 5.1 and Section 9 of \cite{LY}). In the current paper, we see that symmetry does play its role and the effect is embodied by the fact $\mathcal A(k_0)=\mathcal A(-k_0)$, by the form of the eigenvalue equation\eqref{eigen} for $M_k$ where $k$ enters the equation through $k^2$, by the parametrization of center space, during the computation of the dynamics on center manifold (for example, the form of $\Phi(A, A^*)$), by the paragraph below \eqref{pe}. However, the computations of the dynamics on center manifold in the current work did not refer to normal form theory. The reason is that the representation of the linear operator $\mathcal L(a_c,\delta_c)$ on the center space coordinated through the conjugate pair $(A, A^*)$ is the two by two zero matrix $0_{2\times 2}$. We refer the reader to Theorem 6.2 in \cite{LY} or Theorem 7.11 in \cite{Yao} for this point. Though we did not use the normal form theory to help us to do computations, it is interesting to compare the $O(2)$-equivariant Hopf bifurcations in \cite{Yao, LY} and the bifurcations here. In particular, both $\sigma'(0)$ and $\sigma''(0)$ are allowed to be zero in the current paper, which means the important fact that \textit{hyperbolicity} and \textit{genuine nonlinearity} (see page 90 of \cite{Br2} for the two definitions; see also \cite{BD, Lax}) are not necessary for the current study.

\subsection{Choice of spaces } This can be regarded as a remark for Section 2 and a continuation of Subsection 7.1 and we content ourselves in touching only several points among those we can make. In the following discussion, we talk about infinitely dimensional dynamics. There are two main ways of choosing working spaces in the literature: one is a space pair, say $D\subset X$, and the other is a space triplet $Z\subset Y\subset X$ as we adopted. In the former, one usually choose $D$ as the domain of the linear operator on a Banach space $X$ which is properly large. While in the latter, one may choose $Z$ to be a subspace of the domain of the linear operator on the large ambient Banach space $X$ while $Y$ to be the interpolation space which takes care of the nonlinearities in a specific problem. Consequently, bifurcations occurs in $X$ in the former and in $Y$ in the latter. Obviously, the former way of choosing working spaces can be regarded as a special case of the latter one. It is due to this elementary fact that the space triplet choice way usually locates more precisely the spaces in which bifurcation dynamics occur. Carefully checking the proofs in the center manifold theory, we also see that the theory based on space pairs is more effective if the linear operators involved are sectorial (which is the case for dissipative partial differential equations) while the theory established with space triplets can deal with linear operators that are not sectorial. Hence the latter has more applications besides dissipative partial different equations. \\



\begin{thebibliography}{10}

\bibitem{Ba} 

J.M. Ball,  \emph{Some open problems in elasticity. In Geometry, Mechanics, and Dynamics}, pages 3--59, Springer, New York, 2002.

\bibitem{BB}
S. Bianchini, A. Bressan, \emph{Vanishing viscosity solutions
of nonlinear hyperbolic systems}, Annals of Mathematics, \textbf{161}
(2005), 223Ð342.

\bibitem{Br1} 
A. Bressan, \emph{A tutorial on the center manifold
theorem}, from website of A. Bressan:
www.math.psu.edu/bressan/PSPDF/cmanif.pdf

\bibitem{Br2}
A. Bressan, \emph{Hyperbolic Systems of Conservation Laws: The One-Dimensional Cauchy Problem}, Oxford Lecture Series in Mathematics and Its Applications,
Hardcover, 264 pages, Oxford University Press, USA, 2000.

\bibitem{BD}
S. Benzoni-Gavage, D. Serre,
\emph{Multi-dimensional Hyperbolic Partial Differential Equations: First-order Systems and Applications}, Oxford Mathematical Monographs,
Hardcover, 536 pages, Oxford University Press, 2007.

\bibitem{Ca} 
J. Carr, \emph{Applications of centre manifold theory},
Applied math series 35, Springer, 1981, +172 pages.

\bibitem{Da}
C. Dafermos, \emph{Hyperbolic Conservation Laws in Continuum Physics}, Grundlehren der mathematischen Wissenschaften Series,
Springer 2010, 3rd ed, 710 pages.

\bibitem{HI} 
M. Haragus, G. Iooss, \emph{Local bifurcations, center
manifolds, and normal forms in infinite-dimensional dynamical
systems}, Universitest, Springer, 2010, 1st edition, +339 pages.

\bibitem{Ch} 
C.C. Chicone, \emph{Ordinary differential equations with
applications}, Texts in applied math, Springer 2010, 2nd edition,
+658 pages.

\bibitem{EW} 
J.N. Elgin, X. Wu, \emph{Stability of cellular states of the Kuramoto-Sivashinsky equation}, SIAM J. Appl. Math. , 56 (1996) pp. 1621Ð1638.

\bibitem{FW}
A. Fetter, J. Walecka, \emph{
Theoretical Mechanics of Particles and Continua},
Dover Books on Physics Series, Dover Publications, 2003, 592 pages

\bibitem{Go} 
J. Goodman, \emph{Stability of the Kuramoto-Sivashinsky and related systems}, Commun. Pure Appl. Math. , 47 (1994) pp. 293Ð306.

\bibitem{GSS} 
\emph{M. Golubitsky, I. Stewart, D.G. Schaeffer, Singularities and Groups in Bifurcation Theory}, Vol 2, Applied Math Sciences 69, Springer-Verlag, 1985.

\bibitem{He} 
D. Henry, \emph{Geometric theory of semilinear parabolic
equations}, Lecture notes in math, Springer, 1993, + 352 pages.

\bibitem{Ho}
E. Hopf, \emph{A mathematical example displaying features of turbulence}, Communications on Appl. Math. 1, (1948), 303Ð322.

\bibitem{IA} 
G. Iooss, M. Adelmeyer. \emph{Topics in bifurcation theory and applications}, Second ed. Advanced
Series in Nonlinear Dynamics, 3. World Scientific Publishing, River Edge, NJ, 1998.

\bibitem{JS1}
H. Jia, V. Sverak, \emph{Are the incompressible 3d Navier-Stokes equations locally ill-posed in the natural energy space?}  \textbf{arXiv:1306.2136}.

\bibitem{JS2}
H. Jia, V. Sverak, \emph{Local-in-space estimates near initial time for weak solutions of Navier-Stokes equations and forward self-similar solutions}, 
Invent.Math, Online DOI: 10.1007/s00222-013-0468-x, 2013.

\bibitem{KS} 
E.F. Keller, L.A. Segel, \emph{Travelling bands of chemotactic bacteria: a theoretical analysis}, J. Theor. Biol, 1971
(30), 235-248.

\bibitem{KT} 
Y. Kuramoto, T. Tsuzuki, \emph{Persistent propagation of concentration waves in dissipative media far from thermal equilibrium}, Progr. Theoret. Phys. , 55 (1976) pp. 356Ð369.

\bibitem{Lax} P.D. Lax
\emph{Hyperbolic Systems of Conservation Laws II},  CPAM, 10 (1957), 537-566.

\bibitem{LY} L. Tong, J. Yao,
\emph{Equivariant Hopf bifurcation with arbitrary pressure laws in continuum mechanics}, Submitted, see also:  \textbf{arXiv:1312.4248}.


\bibitem{MK} 
J. Moehlis,  E. Knobloch, \emph{Equivariant bifurcation theory}, Scholarpedia, 2(9):2511, 2007 (doi:10.4249/scholarpedia.2511)

\bibitem{MW1} 
T. Ma, S. Wang, \emph{Bifurcation Theory and Applications}, World Scientific Series on Nonlinear Science: Seris A (Book 53), World Scientific Publishing Company, 2005, +392 pages.

\bibitem{MW2}
T. Ma, S. Wang, \emph{Stability and Bifurcation for Nonlinear Evolutionary Equations}, Scientific Publishing House of China, 2007, +441 pages.

\bibitem{NS}
K. Nakanishi, W. Schlag, \emph{Invariant Manifolds and Dispersive Hamiltonian Evolution Equations}, Zurich Lectures in Advanced Mathematics, 2011, 258 pages, softcover.

\bibitem{Ni}
T. Nishida, \emph{Nonlinear hyperbolic equations and related topics in fluid dynamics}, Publications Math$\acute{e}$matique $d'$Orsay 78-02, Universit$\acute{e}$ de Paris-Sud, D$\acute{e}$partement de Math$\acute{e}$matique, 1978, 246 pages.

\bibitem{PYZ} A. Pogan, J. Yao, K. Zumbrun,
\emph{O(2) Hopf bifurcation of viscous shock waves in a channel}, Physica D, in press, doi:10.1016/j.physd.2015.03.002.

\bibitem{Ra} L. Rayleigh,
\emph{On convection currents in a horizontal layer of fluid, when the higher temperature is on the under side}, Philosophical Magazine, Vol XXXII, pp. 529-546, 1916.

\bibitem{Si} 
G. Sivashinsky, \emph{Nonlinear analysis of hydrodynamic instability in laminar flames I. Derivation of basic equations}, Acta Astron. , 4 (1977) pp. 1177Ð1206.

\bibitem{SM} G. Sivashinsky, D. Michelson, 
\emph{On irregular wavy flow of a liquid film down a vertical plane}, Progr. Theoret. Phys. , 63 (1980) pp. 2112Ð2114.

\bibitem{TZ1} B. Texier, K. Zumbrun, 
\emph{Relative Poincare-Hopf
bifurcation and galloping instability of travelling waves}, Methods
and Applications of Analysis, 2005 (12) 349-380.

\bibitem{TZ2} B. Texier, K. Zumbrun, \emph{Hopf Bifurcation of Viscous Shock Waves in Compressible Gas Dynamics and
MHD,} Archive for Rational Mechanics and Analysis 2008 (190)107-140.

\bibitem{Yao}
J. Yao, \textit{$O(2)$-Hopf bifurcation for a model of cellular shock instability}, Physica D, 269 (2014), 63-75.

\end{thebibliography}
\end{document}